\documentclass{jgcc}

\pdfoutput=1
\usepackage{lastpage}
\jgccdoi{15}{1}{2}{11728} 
\jgccheading{}{\pageref{LastPage}}{}{}{Mar.~28,~2022}{Sept.~15,~2023}{}   

\usepackage[utf8]{inputenc}

\usepackage{graphicx}
\usepackage{amssymb}
\usepackage{amsmath}

\usepackage{mathtools}

\usepackage{leftidx}

\usepackage{amsthm}
\newtheorem{theorem}{Theorem}[section]
\newtheorem{proposition}[theorem]{Proposition}
\newtheorem{lemma}[theorem]{Lemma}

\newtheorem{remark}[theorem]{Remark}
\newtheorem{corollary}[theorem]{Corollary}
\newtheorem{question}[theorem]{Question}
\newtheorem{conjecture}[theorem]{Conjecture}

\DeclareMathOperator{\GL}{GL}

\let\Im\relax
\DeclareMathOperator{\Im}{Im}
\DeclareMathOperator{\supp}{supp}
\DeclareMathOperator{\rng}{rng}

\DeclareMathOperator{\lcm}{lcm}

\DeclareMathOperator{\Conj}{Conj}

\DeclareMathOperator{\CD}{CD}

\DeclareMathOperator{\Tor}{Tor}

\DeclareMathOperator{\Sym}{Sym}
\DeclareMathOperator{\Ball}{B}

\DeclareMathOperator{\id}{id}

\DeclareMathOperator{\Max}{max}

\DeclareMathOperator{\normfileq}{\unlhd_{\text{f.i.}}}

\title[Bounding conjugacy depth functions for wreath products]{Bounding conjugacy depth functions for wreath products of finitely generated abelian groups}
\author[M. Ferov and M. Pentigore]
{Michal Ferov and Mark Pengitore}

\thanks{The first author is currently support by the Australian Research Council Laureate Fellowship FL170100032 of Professor George Willis. The second author is partially supported by NSF grant DMS-1812028 ``Geometry, Topology, and Dynamics of Spaces of Non-positive Curvature"}

\begin{document}

\begin{abstract}
    In this article, we study the asymptotic behaviour of conjugacy separability for wreath products of abelian groups. We fully characterise the asymptotic class in the case of lamplighter groups and give exponential upper and lower bounds for generalised lamplighter groups. In the case where the base group is infinite, we give superexponential lower and upper bounds. We apply our results to obtain lower bounds for conjugacy depth functions of various wreath products of groups where the acting group is not abelian.
\end{abstract}
\maketitle

\tableofcontents
\newpage

\section{Introduction}
\label{section:intro}
Studying infinite, finitely generated groups through their finite quotients is a common method in group theory. Groups in which one can distinguish elements using their finite quotients are called residually finite. Formally speaking, a group $G$ is said to be \emph{residually finite} if for every pair of distinct elements $f,g \in G$ there exists a finite group $Q$ and a surjective homomorphism $\pi \colon G \to Q$ such that $\pi(f) \neq \pi(g)$ in $Q$. Group properties of this type are called \emph{separability properties} and are usually defined by what types of subsets we want to distinguish. In this article, we study \emph{conjugacy separability}, meaning that we will study groups in which one can distinguish conjugacy classes using finite quotients. To be more specific, a group $G$ is said to be conjugacy separable if for every pair of nonconjugate elements $f,g \in G$ there exists a finite group $Q$ and a surjective homomorphism $\pi \colon G \to Q$ such that $\pi(f)$ is not conjugate to $\pi(g)$ in $Q$.

\subsection{Motivation} $\:$ \newline
One of the original reasons for studying separability properties in groups is that they provide an algebraic analogue to decision problems in finitely presented groups. To be more specific, if $S \subseteq G$ is a separable subset such that $S$ is recursively enumerable and where one can always effectively construct the image of $S$ under the canonical projection onto a finite quotient of $G$, then one can then decide whether a word in the generators of $G$ represents an element belonging to $S$ simply by checking finite quotients. Indeed, it was proved by Mal'tsev \cite{malcev}, adapting the result of McKinsey \cite{mckinsey} to the setting of finitely presented groups, that the word problem is solvable for finitely presented, residually finite groups in the following way. Given a finite presentation $\langle X \mid R \rangle$ and a word $w \in F(X)$ where $F(X)$ is the free group with the generating set $X$, one runs two algorithms in parallel. The first algorithm enumerates all the products of conjugates of the relators and their inverses and checks whether $w$ appears on the list. On the other hand, the second algorithm enumerates all finite quotients of $G$ and checks whether the image of the element of $G$ represented by $w$ is nontrivial. In other words, the first algorithm is looking for a witness of the triviality of $w$ whereas the second algorithm is looking for a witness of the nontriviality of $w$. Using an analogous approach, Mostowski \cite{mostowski} showed that the conjugacy problem is solvable for finitely presented, conjugacy separable groups. In a similar fashion, finitely presented, LERF groups have solvable generalised word problem meaning that the membership problem is uniformly solvable for every finitely generated subgroup. In general, algorithms that involve enumerating finite quotients of an algebraic structure are sometimes called algorithms of \textbf{Mal'tsev-Mostowski type} or \textbf{McKinsey's} algorithms.

Given an algorithm, it is natural to ask how much computing power is necessary to produce an answer. In the case of algorithms of Mal'tsev-Mostowski type, one can measure their space complexity by the associated depth functions which we go into more detail. Given a residually finite group $G$ with a finite generating set $S$, its residual finiteness depth function $\text{RF}_{G,S} \colon \mathbb{N} \to \mathbb{N}$ quantifies how deep within the lattice of normal subgroups of finite index of $G$ one needs to look to be able to decide whether or not a word of length at most $n$ represents a nontrivial element. In particular, if $w$ is a word in $S$ of length at most $n$, then either $G$ has a finite quotient of size at most $\text{RF}_{G,S}(n)$ in which the image of the element represented by $w$ is nontrivial, and if there is no finite quotient of size less than or equal to $\text{RF}_{G,S}(n)$ in which the image of $w$ is nontrivial, then $w$ must represent the trivial word. In particular, we see that the the residual finiteness depth function of $G$ with respect to the generating set $S$ fully determines the size of finite quotients that McKinsey's algorithm needs to generate in order to give produce an answer. Since every finite group can be fully described by its Cayley table, we see that the space complexity of the word problem of $G$ with respect to the generating set $S$ can be bounded from above by $(\text{RF}_{G,S}(n))^2$. Moreover, the notion of depth function can be generalised to different separability properties. In this note, we study conjugacy separability depth functions which denote as $\Conj_{G,S}(n)$ which is a  function that measures how deep within the lattice of normal subgroups of finite index one needs to go in order to be able to distinguish distinct conjugacy classes of elements of word length at most $n$ with respect to the finite generating subset $S$. Just like in computational complexity, we study these functions up to asymptotic equivalence. See subsection \ref{subsection:depth_fuctions} for the precise definitions of depth functions and the corresponding asymptotic notions.

\subsection{Statement of the results} $\:$ \newline
Not much is known about the asymptotic behaviour of the function $\Conj_{G,S}(n)$ for different classes of groups. The first result of this kind was by Lawton, Louder, and McReynolds \cite{LLM} who showed that if $G$ is a nonabelian free group or the fundamental group of a closed oriented surface of genus $g \geq 2$, then $\Conj_G(n) \preceq n^{n^2}$. For the class of finitely generated nilpotent groups, the second named author and Der\'{e} \cite{Dere_Pengitore} showed that if $G$ is a finite extension of a finitely generated abelian group, then $\Conj_G(n) \preceq (\log(n))^d$ for some natural number $d$, and when $G$ is a finite extension of a finitely generated nilpotent group that is not virtually abelian, then there exist natural numbers $d_1$ and $d_2$ such that $n^{d_1} \preceq \Conj_G(n) \preceq n^{d_2}$. Finally, in \cite{ferov_pengitore}, the authors of this note gave upper bounds for $\Conj_{A \wr B}(n)$ of wreath products of conjugacy separable groups $A$ and $B$ which generalises Remeslennikov's classification of conjugacy separable wreath products \cite{remeslennikov}. However, when applied directly to wreath products of abelian groups, the formulas given in \cite{ferov_pengitore} produce rather coarse upper bounds. Applying \cite[Theorem C]{ferov_pengitore} to the lamplighter group $\mathbb{F}_2 \wr \mathbb{Z}$, one can then demonstrate that its conjugacy depth function can be bounded from above by the function $2^{n^{n^2}}$. Similarly, applying \cite[Theorem C]{ferov_pengitore} to the group $\mathbb{Z} \wr \mathbb{Z}$ one can show that the conjugacy depth function is bounded from above by $n^{n^{n^2}}$. In this note, we only focus on conjugacy depth functions of wreath products of finitely generated abelian groups. This restriction allows us to use more effective methods to obtain much better upper bounds than those presented in \cite{ferov_pengitore}. Additionally, we are able to use methods from commutative algebra to produce lower bounds and in the case of the lamplighter group, we fully determine the asymptotic equivalence class of its conjugacy depth function.

Before stating our main results, we introduce some notation. Letting $f,g \colon \mathbb{N} \to \mathbb{N}$ be non-decreasing functions, we write $f \preceq g$ if there is a constant $C \in \mathbb{N}$ such that $f(n) \leq C g(Cn)$ for all $n \in \mathbb{N}$. If $f \preceq g$ and $g \preceq f$, we then write $f \approx g$.

This first theorem addresses the asymptotic behavior of conjugacy separability of wreath products of the form $A \wr B$ where $A$ is a finite abelian group and $B$ is an infinite, finitely generated abelian group. For the statements of our theorem, we say the torsion free rank of a finitely generated abelian group $A$ is the largest such natural number $k$ such that $A \cong \mathbb{Z}^k \oplus \Tor(A)$ where $\Tor(A)$ is the subgroup of finite order elements of $A$.
\begin{theorem}\label{main_thm lamplighter}
Let $A$ be a finite  abelian group, and suppose that $B$ is an infinite finitely generated abelian group. If the torsion free rank of $B$ is $1$, then
$$
\Conj_{A \wr B}(n) \approx 2^{n}.
$$
If the torsion free rank of $B$ is $k \geq 2$, then
$$
2^{n} \preceq \Conj_{A \wr B}(n) \preceq 2^{n^{2k}}.
$$

\end{theorem}
As a corollary, we are able to compute the precise asymptotic behaviour of conjugacy separability for lamplighter groups.
\begin{corollary}
Let $\mathbb{F}_q$ be the finite field with $q$ elements. Then
  \begin{displaymath}
      \Conj_{\mathbb{F}_q \wr \mathbb{Z}}(n) \approx 2^n.
  \end{displaymath}
\end{corollary}

This next theorem addresses the asymptotic behavior of conjugacy separability of $A \wr B$ when $A$ and $B$ are both infinite, finitely generated abelian groups.
 
\begin{theorem}\label{main_thm infinite}
Let $A$ be an infinite, finitely generated abelian group, and suppose that $B$ is an infinite finitely generated abelian group. If $B$ has torsion free rank $1$, then
$$
(\log(n))^{n} \preceq \Conj_{A \wr B}(n) \preceq (\log(n))^{n^{2}}.
$$
If $B$ has torsion free rank $k > 1$, then
$$
(\log(n))^{n} \preceq \Conj_{A \wr B}(n) \preceq (\log(n))^{n^{2k+2}}.
$$
\end{theorem}

By combining Theorem \ref{main_thm lamplighter} and Theorem \ref{main_thm infinite}, the following corollary gives the best known result for asymptotic behaviour of conjugacy separability of wreath products of finitely generated abelian groups with an infinite acting group.
\begin{corollary}
\label{corollary:main_corrolary}
Let $A$ be a finitely generated abelian group, and suppose that $B$ is an infinite, finitely generated abelian group.
     
Suppose that $A$ is a finitely generated abelian group. If $B$ has torsion free rank $1$, then
$$
(\log(n))^{n} \preceq \Conj_{A \wr B}(n) \preceq (\log(n))^{n^{2}}.
$$
If $B$ has torsion free rank $k > 1$, then
$$
(\log(n))^{n} \preceq \Conj_{A \wr B}(n) \preceq (\log(n))^{n^{2k+2}}.
$$

Suppose that $A$ is finite. If $B$ has torsion free rank $1$, then
$$
\Conj_{A \wr B}(n) \approx 2^n
$$
If $B$ has torsion free rank $k > 1$, then
$$
2^{n} \preceq \Conj_{A \wr B}(n) \preceq 2^{n^{2k}}.
$$
\end{corollary}

This last theorem applies Theorem \ref{main_thm lamplighter} and Theorem \ref{main_thm infinite} to provide exponential lower bounds for conjugacy separable wreath products $A \wr B$ where $\mathbb{Z} \leq Z(B)$ or where $B$ has an infinite cyclic subgroup as a retract.

\begin{theorem}
    \label{theorem:acting with cyclic retract central subgroup}
    Let $A$ be a nontrivial finitely generated abelian group, and suppose that $G$ is a conjugacy separable finitely generated group with separable cyclic subgroups that contains an infinite cyclic group as a retract or satisfies $\mathbb{Z} \leq Z(B)$. If $A$ is finite, then
$$
2^n \preceq \Conj_{A \wr G}(n).
$$
Otherwise,
$$
(\log(n))^{n}  \preceq \Conj_{A \wr G}(n).
$$ 
\end{theorem}

\subsection{Outline of the paper} $\:$ \newline
In Section \ref{section:preliminaries}, we recall standard mathematical notions and concepts that will be used throughout the paper. In particular, in subsection \ref{subsection:depth_fuctions}, we recall the notions of word length, depth functions and associated asymptotic notions. In subsection \ref{subsection:wreath products}, we recall the basic terminology of wreath products of groups. Finally, in subsection \ref{subsec:laurent}, we recall the notion of Laurent polynomial rings and show that groups of the form $R \wr \mathbb{Z}$, were $R$ is a commutative ring, can be realised as $R[x,x^{-1}] \rtimes \mathbb{Z}$ where $R[x,x^{-1}]$ is the ring of Laurent polynomials over the $R$ and $\mathbb{Z}$ acts on $R[x,x^{-1}]$ via multiplication by $x$. We finish this section by giving a criterion for conjugacy for such groups purely in terms of commutative algebra.

In Section \ref{sec:lower bounds} we use methods from commutative algebra to produce lower bounds for the conjugacy depth functions by constructing infinite sequences of pairs of  non-conjugate elements that require quotients that are at least exponentially large in the word lengths of the nonconjugate pair of elements in order to remain non-conjugate.

In Section \ref{sec:upper bounds}, we use combinatorial methods together with the conjugacy criterion for wreath products of abelian groups to construct upper bounds for wreath products of abelian groups.

Finally, in Section \ref{sec:proofs} we combine the lower bounds obtained in Section \ref{sec:lower bounds} together with the upper bounds constructed in Section \ref{sec:upper bounds} to prove Theorem \ref{theorem:acting with cyclic retract central subgroup}. We then proceed to apply our methods to give lower bounds on the conjugacy depth function for wreath products where the acting group may not necessarily be abelian.

\section{Preliminaries}
\label{section:preliminaries}
We denote $\mathbb{F}_{p}$ as the finite field of $p$ elements where $p$ is prime. We denote $\Sym(n)$ as the symmetric group on $n$ letters. For $x,y \in G$, we write $x \sim_G y$ if there exists an element $z \in G$ such that $zxz^{-1} = y$ and suppress the subscript when the group $G$ is clear from context. Whenever the given group is abelian, we will use additive notation.

We say that a subgroup $H \leq G$ is \textbf{conjugacy embedded} in $G$ if for every $f,g \in H$ we have that $f\sim_H g$ if and only if $f \sim_G g$. Following the definition, one can easily check that the relation of being conjugacy embedded is transitive. That means if $A \leq B \leq C$ such that $A$ is conjugacy embedded in $B$ and $B$ is conjugacy embedded in $C$, then $A$ is conjugacy embedded in $C$.

Given a group $G$, we say that a subgroup $R \leq G$ is a \textbf{retract} of $G$ if there exists a surjective homomorphism $\rho \colon G \to R$ such that $\rho \restriction_{R} = \id_R$. The following remark is a natural consequence of the definition of being a retract.
\begin{remark}
\label{remark:retracts are conjugacy embedded}
    Let $G$ be a group, and let $R \leq G$ be a subgroup. If $R$ is a retract of $G$, then $R$ is conjugacy embedded in $G$
\end{remark}

The next lemma allows us to reduce the study of conjugacy in a semidirect product of abelian groups $A \rtimes B$ to conjugacy in $A \rtimes (B / K)$ where $K$ is the kernel of the action of $B$ on $A$. Since wreath products are a special type of a semidirect product, this lemma will be useful throughout the article. Finally, in this lemma, we will be using additive notation, with $B$ acting on $A$ by multiplication, i.e. for $b_1, b_2 \in B$ and $a\in A$ we write
\begin{align*}
        b_1 \cdot b_2 \cdot a_1 &= (b_1 + b_2)\cdot a_1,\\
        0 \cdot a &= a\\
        b_1 \cdot (-a) &= - (b_1 \cdot a) = -b_1 \cdot a.
\end{align*}
In particular, for $a_1, a_2 \in A$ and $b_1, b_2 \in B$, we write
\begin{align*}
    (a_1, b_1)(a_2, b_2) &= (a_1 + b_1 \cdot a_2, b_1 + b_2),\\
    (a_1, b_1)(a_2, b_2) &= (a_1 + b_1 \cdot a_2, b_1 + b_2),\\
    (a_1, b_1)^{-1} &= (-(-b_1) \cdot a_1, -b_1) = ((-b_1)(-a_1), -b_1)
\end{align*}
\begin{lemma}
\label{lemma:conjugacy through kernel}
Suppose that $A$ and $B$ are finitely generated abelian groups, and suppose that $a_1, a_2 \in A$ and $b \in B$. Then 
$$
(a_1,b) \sim_{A \rtimes B} (a_2, b) \quad \text{ if and only if } \quad (a_1, \bar{b}) \sim_{A \rtimes (B/K)} (a_2, \bar{b}),
$$
where $K$ is the kernel of the action of $B$ on $A$.
\end{lemma}
\begin{proof}
For this proof, we denote the action of $b \in B$ on $a \in A$ as $b \cdot a$. We note that if $(a_1, b) \sim_{A \rtimes B} (a_2, b)$, then clearly $(a_1, \bar{b}) \sim_{A \rtimes (B/K)} (a_2, \bar{b}).$ Thus, we may assume that $(a_1, \bar{b}) \sim_{A \rtimes (B/K)} (a_2, \bar{b})$. Suppose that there exists $(x,y) \in A \rtimes B$ and $k \in K$ such that 
$$
(x,y) (a_1,b) (x,y)^{-1} = (0,k) (a_2,b) .
$$ Thus, we have
\begin{eqnarray*}
(x,y) (a_1,b) (x,y)^{-1} &=& (0,k) (a_2,b) \\
(x + y\cdot a_1, y + b) ((-y)\cdot (-x), -y) &=& (k \cdot a_2, k+b) \\
(x + y\cdot a_1 + (y + b)\cdot( -y)\cdot (-x), y + b -y) &=& (a_2, k+b) \\
(x + y\cdot a_1 + (y + b - y)\cdot (-x), y + b -y) &=& (a_2, k+b) \\
(x + y\cdot a_1 - b \cdot x, b) &=& (a_2, k+b). \\ 
\end{eqnarray*}
Hence, we must have that $k = 0$. Therefore, we have $(x,y)(a_1,b)  (x,y)^{-1} = (a_2,b)$ giving our claim.
\end{proof}

Given an abelian group $B$, we will use $\Tor(B)$ to denote
$$
\Tor(B) =\{ b \in B: b^m = 1 \text{ for some } m \in \mathbb{Z} \}.
$$
When $B$ is a finitely generated, it it easy to see that $\Tor(B)$ is a characteristic subgroup which provides a splitting 
\begin{displaymath}
    B = \mathbb{Z}^k \oplus \Tor(B)    
\end{displaymath}
for some $k \in \mathbb{N}$. We say that $k$ is the \textbf{torsion-free rank} of $B$. By fixing a splitting, we define $\tau \colon B \to \Tor(B)$ and $\phi \colon B \to \mathbb{Z}^k$ as the associated  retractions. Then every element $b \in B$ can be uniquely expressed as $b = \tau(b) + \phi(b)$. We say that $\tau(b)$ is the \textbf{torsion part} of $b$ and $\phi(b)$ is the \textbf{torsion-free} part of $b$.
Whenever we say the torsion part or torsion-free part of an element of $B$, we are saying that with respect to some fixed splitting of the above form.

To ease notation, we will view direct sums of groups over some indexing set as finitely supported functions on the indexing set with range in the index groups. More precisely, if 
$$
G = \bigoplus_{i \in I} A
$$ is a direct sum of copies of a group $A$ indexed by a set $I$, then for $f \in G$ we will write $f(i)$ to denote the $i$-th coordinate of $f$. In particular, elements in $G$ correspond to functions $f \colon \mathcal{I} \to A$ where $f(i) = 1$ for all but finitely many elements in $\mathcal{I}$. The \textbf{support} of $f$ which is the set of elements on which $f$ is not trivial will be denoted as
\begin{displaymath}
    \supp(f) = \{i \in I \mid f(i) \neq 1\}.
\end{displaymath}
The \textbf{range} of $f$ will be denoted as
\begin{displaymath}
    \rng(f) = \{f(i) \mid i \in I\}.
\end{displaymath}

\subsection{Asymptotic notions and depth functions} $\:$ \newline
\label{subsection:depth_fuctions}
Let $G$ be a finitely generated group equipped with a finite generating subset $S$. We define the word length of an element $g \in G$ with respect to $S$ as
\begin{displaymath}
    \|g\|_S = \min \{ |w| \mid w \in F(S) \mbox{ and } w =_G g\}.    
\end{displaymath}
where $|w|$ denotes the word length of $w$ in $F(S).$  We use $\Ball_{G,S}(n)$ to denote the ball of radius $n$ centered around the identity with respect to the finite generating subset $S$. When the finite generating subset is clear from context, we will instead write $\Ball_G(n).$


The \textbf{conjugacy separability depth function} of $G$ is defined in the following way. Let $f,g \in G$ be a pair of elements such that $f \not\sim_G g$. The \textbf{conjugacy depth} of the pair $(f,g)$, denoted $\CD_G(f,g)$, is given by
\begin{displaymath}
    \CD_G(f,g) = \min \{ |G/N| \: | \: N \trianglelefteq_{f.i.} G \mbox{ and } fN \not\sim_{G/N} gN \}
\end{displaymath}
with the understanding that $\CD_G(f,g) = \infty$ if no such finite quotient exits. Similar to the definition of residual finiteness, we say that $G$ is \textbf{conjugacy separable} if $\CD_G(g,h) < \infty$ for all $f,h \in G$ such that $f \nsim_G g.$  Given a finite generating subset $S \subseteq G$ for a conjugacy separable group $G$, the \textbf{conjugacy separability depth function} $\Conj_{G,S} \colon \mathbb{N} \to \mathbb{N}$ is defined as
\begin{displaymath}
\Conj_{G,S}(n) = \Max\{ \CD_G(f,g) \mid f,g \in \Ball_{G,S}(n) \mbox{ and } f\not\sim_G g\}.
\end{displaymath}

We note that $\Conj_{G,S}(n)$ depends on the choice of the finite generating subset $S$. However, one can easily check that the asymptotic behaviour does not.  It is well known that a change of a finite generating subset is a quasi-isometry. In particular, if $S_1, S_2 \subset G$ are two finite generating subsets of a group $G$, then $\|\cdot\|_{S_1} \approx \|\cdot\|_{S_2}$. The same holds for depth functions. For non-decreasing functions $f,g \colon \mathbb{N} \to \mathbb{N}$, we write $f \preceq g$ if there is a constant $C \in \mathbb{N}$ such that $f(n) \leq C g(Cn)$ for all $n \in \mathbb{N}$, and if $f \preceq g$ and $g \preceq f$, we then write $f \approx g$. When $G$ is conjugacy separable, we have $\Conj_{G, S_1}(n) \approx \Conj_{G, S_2}(n)$; see \cite{LLM} for more details. As we are only interested in the asymptotic behaviour of the above defined functions, we will suppress the choice of generating subset whenever we reference the depth functions or the word-length.

Let $G$ be a conjugacy separable group with a finitely generated conjugacy embedded subgroup $R$. We now relate the conjugacy depth of a pair of nonconjugate elements $r_1, r_2 \in R$ as elements of $R$ with the conjugacy depth of $r_1,r_2$ as elements of $G$.
\begin{lemma}\label{conj_embedded_conj_depth}
    Let $G$ be a finitely generated conjugacy separable group with a finitely generated conjugacy embedded subgroup $R \leq G$. Then for $r_1, r_2 \in R$ where $r_1 \nsim_R r_2$, we have
    $$
    \CD_R(r_1, r_2) \leq \CD_G(r_1, r_2).
    $$
\end{lemma}
\begin{proof}
Suppose that $r_1 \nsim_R r_2$ for $r_1, r_2 \in R.$ Since $R$ is conjugacy embedded into $G$, we have that $r_1 \nsim_G r_2$. hus, we will show that $\CD_R(r_1, r_2) \leq \CD_G(r_1, r_2)$. Suppose that $N_R \trianglelefteq_{f.i} G$ realises $\CD_G(r_1, r_2)$. Since $\varphi|_R \colon R \to \varphi(R) \leq G/N_G$, we see that $\varphi|_R(r_1) \nsim \varphi|_R(r_2).$ Hence, we have
    $$
    \CD_R(r_1, r_2) \leq |\varphi(R)| \leq |G / N_G| = \CD_G(r_1,r_2). \qedhere
    $$
\end{proof}

We conclude this subsection with the following lemma which relates $\Conj_G(n)$ with $\Conj_R(n)$ where $R$ is a retract of a finitely generated conjugacy separable group $G$.
\begin{lemma}
    \label{lemma:retracts lower bound}
    Suppose that $G$ is a finitely generated conjugacy separable group with a finitely generated subgroup $R \leq G$ such that $R$ is a retract. Then $R$ is conjugacy separable and conjugacy embedded into $G$. Moreover, we have
    \begin{displaymath}
        \Conj_R(n) \preceq \Conj_G(n).
    \end{displaymath}
\end{lemma}
\begin{proof}
    Let $\rho \colon G \to R$ be the corresponding retraction. We start by showing there is a finite generating set $X \subseteq G$ such that $X = X_R \dot\cup X_K$, $R = \langle X_R \rangle,$ and $\langle X_K \rangle \cap R = \{1\}$. Suppose that $G = \langle X'\rangle$, where $X = \{x_1, \dots, x_m\}$. We set $X_R = \{\rho(x_1),\dots, \rho(x_m)\}$ and $X_K = \{\rho(x_1)^{-1}x_1, \dots, \rho(x_m)^{-1}x_m\}$. It is straightforward to see that $R = \langle X_R \rangle$ and $\langle X_K \rangle \leq \ker(\rho)$, so $\langle X_K \rangle \cap R = \{1\}$. We also have $\|r\|_{X_R} = \|r\|_X$ for every $r \in R$.
    
    Now suppose that $r_1, r_2 \in \Ball_R(n)$ satisfy $r_1 \not\sim_R r_2$. Following the previous paragraph together with remark \ref{remark:retracts are conjugacy embedded}, we see that $r_1, r_2 \in \Ball_G(n)$ and that $r_1 \not\sim_G r_2$. By Lemma \ref{conj_embedded_conj_depth}, we see that
    $$
\CD_R(r_1, r_2) \leq \CD_G(r_1, r_2).
    $$
    We note that this inequality holds for all $r_1,r_2 \in R$ where $r_1 \nsim_R r_2$, and since $B_R(n) \subset B_G(n)$, it follows that 
    $$
    \Conj_R(n) \leq \Conj_G(n). \qedhere
    $$
\end{proof}

\subsection{Wreath products} $\:$ \newline
\label{subsection:wreath products}
For groups $A$ and $B$, we denote the restricted wreath product of $A$ and $B$, written as $A \wr B$, by
    \begin{displaymath}
        A \wr B = \left(\bigoplus_{b \in B} A \right) \rtimes B.
    \end{displaymath}
where $B$ acts on $\bigoplus_{b \in B} A$ via left multiplication on the coordinates. An element $f \in \bigoplus_{b \in B} A$ is understood as a function $f \colon B \to A$ such that $f(b) \neq 1$ for only finitely many $b \in B$. With a slight abuse of notation, we will use $A^B$ to denote $\bigoplus_{b \in B} A$. The action of $B$ on $A^B$ is then realised as $b \cdot f(x) = f(b x)$. 

Following the given notation, if $H \leq A$ and $K \leq B$, we will use $H^K$ to denote the subset
\begin{displaymath}
    H^K = \{f \in A^B \mid \supp(f) \subseteq K \mbox{ and } \rng(f) \subseteq H\}.
\end{displaymath}
Keeping this notation in mind, the wreath product $H \wr K$ can then be naturally identified with the subgroup $H^K \rtimes K \leq A \wr B$.

\begin{lemma}
    \label{lemma:conjugacy embedded subgroups}
    Let $A,B$ be finitely generated abelian groups, and suppose that $R_A \leq A$ and $R_B \leq B$ are retracts. Then the group $R = R_A \wr R_B$ is a retract of $G = A \wr B$. In particular, $R$ is conjugacy embedded in $G$ and $\Conj_R(n) \preceq \Conj_G(n)$.
\end{lemma}
\begin{proof}
    Let $\rho_A \colon A \to R_A$ and $\rho_B \colon B \to R_B$ be the associated retraction maps. We define a map $\rho \colon A\wr B \to R_A \wr R_B$ in a following way: given $f \in A^B, b\in B$ we set $\rho((f,b)) = \rho(f) \rho_B(b)$ where $\rho(f)$ is a function in $R_A^{R_B}$ defined as
    \begin{displaymath}
        \rho(f)(x R_B) = \rho_A \left(\prod_{y \in x R_b} f(y) \right).
    \end{displaymath}
    We see that $\rho$ is a surjective homomorphism and $\rho|_{R} = \id_{R}$, and thus, it follows that $R$ is a retract of $A \wr B$. We finish by noting that Remark \ref{remark:retracts are conjugacy embedded} implies $R$ is conjugacy embedded in $A \wr B$.
\end{proof}
    
Suppose that $b \in B$ and $f \in A^B$ is a function with a finite support. We say that $f$ is \textbf{minimal} with respect to $b$ if all elements of $\supp(f)$ lie in distinct cosets of $\langle b\rangle$ in $B$. We will say that an element $fb \in A \wr B$ is \textbf{reduced} if $f$ is minimal with respect to $b$.
    
The following lemma is a special case of \cite[Lemma 5.13]{ferov_pengitore}.
\begin{lemma}\label{lemma: lemma5.13}
    \label{lemma:reduced_elements}
    Let $A$, $B$ be finitely generated groups, and suppose that $b \in B$ and $f \colon B \to A$ are given such that $fb \in \Ball_{A \wr B}(n)$. Then there exists a constant $C$ independent of $b$ and $f$ and $f' \in A^B$ such that the following hold:
    \begin{enumerate}
    \item $f'b \sim fb$
    \item $f'b$ is reduced
    \item $\|f'b\| \leq C\|fb\|$.
    \end{enumerate}
\end{lemma}
    
The following statement and its proof which provides a conjugacy criterion for wreath products of abelian groups follows from \cite[Lemma 5.14]{ferov_pengitore}. 
\begin{lemma}
    \label{lemma:abelian_wreath_conjugacy_criterion}
    Let $A,B$ be abelian groups, and let $G = A \wr B$ be their wreath product. Let $f_1, f_2 \in A^B$, $b_1, b_2 \in B$ be such that the elements $f_1b_1$ and $f_2 b_2$ are reduced. Then 
    $$f_1 b_1 \sim_G f_2 b_2 \quad \text{ if and only if } \quad b_1 = b_2 \quad \text{ and } \quad f_1b \in (f_2b)^B.
    $$ 
    In particular, there exists an element $c \in B$ such that 
    $$
    c \supp(f_1) = \supp(f_2) \quad \text{ and } \quad f_1(cx) = f_2(x)
    $$
    for all $x \in B$.
\end{lemma}
One interpretation of Lemma \ref{lemma:abelian_wreath_conjugacy_criterion} is that by ensuring that we are only working with reduced elements of $A \wr B$, we only need to worry about them being conjugate by an element from $B$.

Let $A$ be a finite abelian group and let $B$ be a finitely generated abelian group of torsion free rank at least $1$. This next lemma allows to reduce the study of asymptotic lower bounds for conjugacy separability of groups of the form $A \wr B$ to that of groups of the form $\mathbb{F}_p \wr \mathbb{Z}.$
\begin{lemma}
    \label{lemma:lamplighter conjugacy embedded}
    Let $A$ be a finite abelian group where $p \mid |A|$, and let $B$ be an infinite, finitely generated abelian group. The group $\mathbb{F}_p \wr \mathbb{Z}$ is conjugacy embedded in the group $A \wr B$ and 
    $$
    \Conj_{\mathbb{F}_p \wr \mathbb{Z}}(n) \preceq \Conj_{A \wr B}(n).
    $$
\end{lemma}
\begin{proof}
    Since $\mathbb{Z}$ is a retract of $B$, we have that $A \wr \mathbb{Z}$ is a retract of $A \wr B.$ We then have by Lemma \ref{lemma:retracts lower bound} that $\Conj_{A \wr \mathbb{Z}}(n) \preceq \Conj_{A \wr B}(n)$. Thus, we may assume that $B \cong \mathbb{Z}.$ 
    
    We now demonstrate that $\mathbb{F}_p \wr \mathbb{Z}$ is conjugacy embedded into $A \wr \mathbb{Z}.$ Suppose that $\mathbb{Z} =\left<b\right>.$ Let $f_1, f_2 \colon \mathbb{Z} \to A$ be given such that $f_1 b^{s_1}, f_2b^{s_2} \in \mathbb{F}_p \wr \mathbb{Z}$ satisfy $f_1 b^{s_1} \sim_{A \wr \mathbb{Z}} a^{t_2}b^{s_2}$ where $s_1, s_2 \in \mathbb{Z}.$ Moreover, we may assume are both reduced. We claim that $f_1 b^{s_1} \sim_{\mathbb{F}_p \wr \mathbb{Z}} f_2b^{s_2}$. Since $f_1 b^{s_1} \sim_{A \wr \mathbb{Z}} f_2b^{s_2}$, we must have $b^{s_1} \sim_{\mathbb{Z}} b^{s_2}$, and given that $\mathbb{Z}$ is abelian, we then have $s = s_1 = s_2.$ By Lemma \ref{lemma:abelian_wreath_conjugacy_criterion}, we have there exists a $b^t$ such that $b^t \cdot \supp(f_1) = \supp(f_2)$ and $f_1(b^tx) = f_2(x)$. However,  that is equivalent to $f_1 b^{s}$ and $f_2b^{s}$ being conjugate in $\mathbb{F}_p \wr \mathbb{Z}$ as desired.

    For the second part of the statement, we first show that $\mathbb{F}_p \wr \mathbb{Z}$ can be realised as an undistorted subgroup of $A \wr B$. If $\mathbb{Z}/p^e\mathbb{Z} = \langle a \rangle$, we then see that $\mathbb{Z}/p\mathbb{Z} = \langle a^{p^{e-1}} \rangle$. Letting $X_{p_e} =  \{ a, a^{p^{e-1}} b \} \subseteq L_{p^e}$ and $X_p = \{a^{p^{e-1}}, b\} \subseteq L_{p}$, it then follows that $L_{p^e} = \langle X_{p_e} \rangle$ and $L_p = \langle X_p \rangle$. One can easily check that for any $x \in L_p$ that $\|x\|_{X_p} = \|x\|_{X_{p^e}}$, and subsequently, $\Ball_{L_p, X_p} (n) \subseteq \Ball_{L_{p^e}, X_{p^e}}(n)$. 
    
     Now suppose that $x,y \in \Ball_{\mathbb{F}_p \wr \mathbb{Z}}(n)$ are not conjugate. We then have that $f, g \in \Ball_{A \wr \mathbb{Z}}(n)$, and since $\mathbb{F}_p \wr \mathbb{Z}$ is conjugacy embedded into $A \wr B$, Lemma \ref{conj_embedded_conj_depth} implies
     $$
\CD_{\mathbb{F}_p \wr \mathbb{Z}}(x,y) \leq \CD_{A \wr B}(x,y).
     $$
    As a consequence of the above inequality, and the previous paragraph, we see that 
    $$
    \Conj_{\mathbb{F}_p \wr \mathbb{Z}}(n) \preceq \Conj_{A \wr \mathbb{Z}}(n).\qedhere
    $$
    \end{proof}

The next lemma, which is a direct consequence of \cite[Theorem 3.4]{davis_olshanskii}, relates the size of the support of its function part and the size of the elements in the range of the function with the word length of an element. We omit the proof in order to avoid having to introduce more technical notation, we encourage a curious reader to inspect \cite[Theorem 3.4]{davis_olshanskii} and prove check that the statement indeed holds.
\begin{lemma}
    \label{lemma:bounding_support_and_image}
    Let $A, B$ be finitely generated groups and let $G = A \wr B$ be their wreath product. Then there exists a constant $C > 0$ such that if $g = f b$ where $f \in A^B$ and $b \in B$, then 
    \begin{itemize}
        \item[(i)] $\supp(f) \subseteq \Ball_B(C \|g\|)$,
        \item[(ii)] $\rng(f) \subseteq \Ball_A(C \|g\|)$,
        \item[(iii)] $b \in \Ball_B(C\| g\|)$.
    \end{itemize}
\end{lemma}

Given a wreath product $A \wr B$ with a surjective homomorphism $\pi \colon B \to \overline{B}$, we denote $\tilde{\pi} \colon A \wr B \to A \wr \overline{B}$ as the canonical extension of $\pi$ to all of $A \wr B$ given by
\begin{displaymath}
    \tilde{\pi}(f)(b K) = \sum_{k \in K}f(b + k), 
\end{displaymath}
where $K = \ker(\pi)$. Note that since the group $A$ is abelian and the function $f \in A^B$ is finitely supported, the above sum is well defined. Similarly, if $\pi \colon A \to \overline{A}$ is a surjective homomorphism, we let $\tilde{\pi} \colon A \wr B \to \overline{A} \wr B$ as the natural extension of $\pi$ to all of $A \wr B$.

\subsection{Wreath products and Laurent polynomial rings} $\:$ \newline
\label{subsec:laurent}
Much of the following discussion, which includes undefined notation and terms, can be found in \cite{atiyah_macdonald, eisenbud, lang}. Given a commutative ring $R$, we will write $R[x]$ to denote the ring of polynomials in the variable $x$ with coefficients in $R$, and we will use $R[x, x^{-1}]$ to denote the ring of Laurent polynomials over $R$.

We first note that $R[x, x^{-1}]$ is the localisation of the ring $R[x]$ on the set $S =\{x^m \: | \: m \in \mathbb{N} \}$. We then have that the ideals of $R[x, x^{-1}]$ are in one-to-one correspondence with ideals of $R[x]$ that don't intersect the set $S$. In particular, for any ideal $\mathcal{I} \subset R[x]$ where $\mathcal{I} \cap S = \emptyset$, we have that $R[x, x^{-1}] / (S^{-1} \mathcal{I}) = S^{-1}(R[x] / \mathcal{I})$. We finish by observing that the maximal ideals of $R[x, x^{-1}]$ can be written as $\mathcal{I} = (f)$ where $f$ is an irreducible polynomial not in $S$. If $k = \deg(f)$, then $|R[x, x^{-1}] / \mathcal{I}| = |R|^k$.

We now focus on the following representation of $R \wr \mathbb{Z}$ as a semidirect product of the ring $R[x, x^{-1}]$ and $\mathbb{Z}$. First, let us define a function $P \colon R^{\mathbb{Z}} \to R[x, x^{-1}]$ given by
\begin{displaymath}
    P(f) = \sum_{m \in \mathbb{Z}} f(m)x^m.
\end{displaymath}
One can easily check in the context of finitely supported functions that $P$ is a bijection and for any $r \in R$, $f,g \in R^{\mathbb{Z}}$, and $m \in \mathbb{Z}$ that the following holds:
\begin{enumerate}
    \item[(i)] $P(rf) = rP(f)$,
    \item[(ii)] $P(f + g) = P(f) + P(g)$,
    \item[(iii)] $P(m\cdot f) = x^m P(f)$.
\end{enumerate}
We will use these three equalities without mention.

\begin{lemma}
    Let $R$ be either the ring $\mathbb{Z}$ or $\mathbb{F}_p$ where $p$ is prime. The group $R \wr \mathbb{Z}$ is isomorphic to $R[x,x^{-1}] \rtimes \mathbb{Z}$ where $R[x,x^{-1}]$ is ring of Laurent polynomials with addition and for $t \in \mathbb{Z}$, we have $t \cdot f(x) = x^tf(x).$
\end{lemma}
\begin{proof}
Let $\varphi \colon R \wr \mathbb{Z} \to R[x,x^{-1}] \rtimes \mathbb{Z}$ be the map given by $\varphi\left(f m\right) = (P(f),m)$. It is then easy to see that this map is an isomorphism. \end{proof}

The following lemma allows us to understand finite quotients of $R \wr \mathbb{Z}$ in terms of the cofinite ideals of $R[x, x^{-1}]$. For the following lemma, we identify $R[x, x^{-1}]$ with the normal subgroup of $R \wr \mathbb{Z}$ given by elements of the form $(P,0)$ where $P \in R[x,x^{-1}].$

\begin{lemma}\label{ideal_normal}
Let $R$ be either the ring $\mathbb{Z}$ or $\mathbb{F}_p$ where $\mathbb{F}_p$ be the field with $p$ elements. Let $N \trianglelefteq R \wr \mathbb{Z}$. Then $N \cap R[x, x^{-1}]$ is an ideal in $R[x, x^{-1}]$. In particular, if $N \trianglelefteq_{f.i} R \wr \mathbb{Z}$, then $N \cap R[x, x^{-1}]$ is a cofinite ideal of $R[x, x^{-1}]$.
\end{lemma}
\begin{proof}
Let $M = N \cap R[x, x^{-1}].$  We note for $(P,0) \in M$ that
$$
\left(0, m  \right)\left(P,0\right) \left( 0, -m\right)=  \left(x^m P,0 \right) \in M
$$
since $M$ is normal. In particular, we have that $M$ is closed under multiplication by $x^m$ in $R[x, x^{-1}]$ for all $m \in \mathbb{Z}$. Additionally,  for $(P_1,0), (P_2,0) \in M$ we have that
$$
(P_1,0) (P_2,0) = (P_1 + P_2,0).
$$
That implies $M$ is closed under addition. Since multiplying $P$ by $r$ is the same as adding $r$ copies of $P$ and given that $M$ is a subgroup of $R[x,x^{-1}]$ with addition as its group operation, we must have that $(rP,0) \in M$. Thus, for $(P,0) \in M$ and a general element $\sum_{m \in \mathbb{Z}} a_m x^m$ of $R[x, x^{-1}]$, we may write
$$
P \cdot \sum_{m \in \mathbb{Z}} a_m x^m = \sum_{m \in \mathbb{Z}} a_m x^m P \in M.
$$
Thus, $M$ is an ideal in $R[x, x^{-1}]$. Moreover, the second part of the statement immediately follows.
\end{proof}

The following lemma gives the explicit expression for the conjugacy class of an arbitrary element of $R \wr \mathbb{Z}$. 
\begin{lemma}\label{conjugacy_class}
    Let $R$ be either the ring $\mathbb{Z}$ or $\mathbb{F}_p$ where $\mathbb{F}_p$ is the field with $p$ elements. For $(P, m) \in R \wr \mathbb{Z}$, its conjugacy class is given by
    $$
    \left\{\left( \left( x^{\ell} P +  \left(x^m - 1 \right)Q, m\right) \: \right| \ell \in \mathbb{Z}, Q \in R[x, x^{-1}]\right\}.
    $$
\end{lemma}
\begin{proof}
Let $Q \in R[x, x^{-1}]$ and $\ell \in \mathbb{Z}$ be arbitrary. We then write 
\begin{align*}
    (Q,\ell)(P, m)(Q,\ell)^{-1} &= \left(Q + x^\ell P , \ell + m \right)\left(-x^{-\ell}Q, -\ell \right)\\
                                &= \left(Q + x^\ell P - x^{\ell + m}x^{-\ell}Q, \ell + m -\ell\right)\\
                                &= \left(Q + x^\ell P - x^mQ,  m \right)\\
                                &= \left(x^\ell P + (1-x^m) Q, m \right).
\end{align*}
Since $Q$ was arbitrary, we may replace it by $-Q$ allowing us to write
\begin{displaymath}
    (Q,\ell)(P, m)(Q,\ell)^{-1} = \left(x^\ell P + (x^{m} - 1) Q, m \right).
\end{displaymath}
From here, our statement is clear.
\end{proof}


\section{Lower bounds}
\label{sec:lower bounds}
In this section, we construct asymptotic lower bounds for conjugacy separability for the groups $\mathbb{F}_p \wr \mathbb{Z}$ and $\mathbb{Z} \wr \mathbb{Z}.$ which we divide into two subsections. The first subsection goes over the lower bounds for $\Conj_{\mathbb{F}_p \wr \mathbb{Z}}(n)$. The second subsection constructs lower bounds for $\Conj_{\mathbb{Z} \wr \mathbb{Z}}(n)$. 

\subsection{Lower bounds for $\Conj_{A \wr B}(n)$ where $A$ is a finite abelian group} $\:$ \newline
 In this section, we provide asymptotic lower bounds for $\Conj_{A \wr B}(n)$ when $A$ is a finite abelian group and $B$ is an infinite, finitely generated abelian group by finding asymptotic bounds for $\Conj_{\mathbb{F}_p \wr \mathbb{Z}}(n)$ .
\begin{proposition}\label{rank_one_lamplighter}
    Let $A$ be a finite abelian group and $B$ be an infinite, finitely generated abelian group. Then $2^n \preceq \Conj_{A \wr B}(n).$
\end{proposition}
\begin{proof}
By Lemma \ref{lemma:lamplighter conjugacy embedded}, we may assume that $A \cong \mathbb{F}_p$ for some prime and that $B \cong \mathbb{Z}$. We need to find an infinite sequence of pairs of elements $\{f_i,g_i\}^{\infty}_{i = 1}$ such that 
\begin{enumerate}
    \item[(i)] $\lim_{i \to \infty} \text{max}\left\{\|f_i\|, \|g_i\|\right\}  = \infty$,
    \item[(ii)] $f_i \nsim g_i$,
    \item[(iii)] $p^{C \text{max}\left\{\|f_i\|, \|g_i\|\right\}} \leq \CD_{\mathbb{F}_p\wr \mathbb{Z}}(f_i,g_i)$,
\end{enumerate}
where $C > 0$ is some constant.

Let $\{q_i\}_{i=1}^\infty$ be an enumeration of the set of primes greater than $p$ such that $p$ is a primitive root mod $q_i$. In this case, it is well known that $\psi_{q_i}(x) = \sum_{i=1}^{q_i-1}x^i$ is an irreducible polynomial over $\mathbb{F}_p$. Let 
\begin{displaymath}
f_i = \left(x^{q_i} - 1, q_i\right) \quad \text{ and } \quad g_i =  \left(x - 1  + x^{q_i} - 1, q_i\right).
\end{displaymath}
Let us consider the quotient $\mathbb{F}_p[x, x^{-1}] / (\psi_{q_i}(x)) \rtimes (\mathbb{Z} / q_i \mathbb{Z}))$ with the associated projection map $\pi_i$. We then see that
\begin{displaymath}
|\mathbb{F}_p[x, x^{-1}] / (\psi_{q_i}(x)) \rtimes (\mathbb{Z} / q_i \mathbb{Z})| = q_i p^{q_i - 1}
\end{displaymath}
and that
\begin{displaymath}
    \pi_i(f_i) = (0,0) \quad \text{ and } \quad \pi_i(g_i) = (x-1, 0) \neq (0,0). 
\end{displaymath}
It follows that $\pi(f_i) \nsim \pi(g_i)$. Subsequently, we see that $f_i$ and $g_i$ are not conjugate in $\mathbb{F}_p \wr \mathbb{Z}$ and that $\CD_{\mathbb{F}_p\wr \mathbb{Z}}(f_i,g_i) \leq q_i p^{q_i -1}$.

To finish, we will demonstrate that $p^{q_i} \leq \CD_{\mathbb{F}_p\wr \mathbb{Z}}(f_i,g_i)$ for all $i$. In other words, we need to show that if $N \trianglelefteq_{f.i.} \mathbb{F}_p \wr \mathbb{Z}$ is given such that  $|(\mathbb{F}_p \wr \mathbb{Z}) / N| < p^{q_i}$, then $f_i \sim g_i \text{ mod } N$. Suppose that such a normal finite index subgroup $N$ is given. We note by Lemma \ref{ideal_normal} that $\mathcal{J}_N = N\cap \mathbb{F}_p[x, x^{-1}]$ is an ideal in $\mathbb{F}_p[x, x^{-1}]$. In particular, $\mathcal{J}_N \rtimes N \cap \mathbb{Z} \leq N$ is a normal subgroup in $\mathbb{F}_p \wr \mathbb{Z}$ such that if $f_i \sim g_i \text{ mod } \mathcal{J}_N \rtimes N \cap \mathbb{Z},$ then $f_i \sim g_i \text{ mod } N$. Thus,for the purpose of the proof, we may assume that $N \cong \mathcal{J} \rtimes t \mathbb{Z}$ for some $t \in \mathbb{N}$. If $\mathcal{J}_N = \mathbb{F}_p[x, x^{-1}]$, then $(\mathbb{F}_p \wr \mathbb{Z}) / N$ is a finite abelian group. In particular, we have that $\mathbb{F}_p[x, x^{-1}] \leq \ker(\pi_N),$ and thus, $\pi_N(f_i) = \pi_N(g_i)$. Hence, we may assume that $\mathcal{J}_N$ is a proper ideal in $\mathbb{F}_p[x, x^{-1}]$. Moreover, we have that $|\mathbb{F}_p[x, x^{-1}] / \mathcal{J}_N| < p^{q_i}.$ 

Since $\mathbb{F}[x, x^{-1}]$ is a localisation of a principal ideal domain, it is also a principal ideal domain. Therefore, there exists a polynomial $P \in \mathbb{F}_p[x]$ such that $\mathcal{J}_N = (P)$. Thus, we note that one of the following cases must hold:
\begin{displaymath}
    \gcd(x^{q_i} - 1, P) =  \begin{cases}
                                x^{q_i} - 1,\\
                                \psi_{q_i},\\
                                x - 1,\\
                                1
                            \end{cases}.
\end{displaymath}
    
Let us first note that $\mathbb{F}_p[x,x^{-1}]/(P) \leq \left(\mathbb{F}_p[x, x^{-1}]\rtimes \mathbb{Z}\right)/N$. We see that we may ignore the first two cases, as in both we have that $ p^{q_i} \leq \left|\mathbb{F}_p \wr \mathbb{Z} / N \right|$.

For the third case, we have that $x - 1 \in \mathcal{J}_N$. Therefore, we have
\begin{eqnarray*}
    \pi_N(f_i)  &=& (x^{q_i} - 1 \text{ mod } \mathcal{J}_N, q_i \text{ mod } t)\\
                &=& ((x - 1)\psi_{q_i}(x) \text{ mod } \mathcal{J}_N, q_i \text{ mod } t)\\
                &=& (0, q_i \text{ mod } t).
\end{eqnarray*}
Similarly, we have
\begin{eqnarray*}
    \pi_N(g_i)  &=& (x-1 + (x - 1)\psi_{q_i}(x) \text{ mod } \mathcal{J}, q_i \text{ mod } t)\\
                &=& (0, q_i \text{ mod } t).
\end{eqnarray*}
Hence, $\pi_N(f_i) = \pi_N(g_i).$

For the last case, we may assume that $\gcd(x^{q_i} - 1, P) = 1$. Let us recall that, following Lemma \ref{conjugacy_class}, we can write the conjugacy class of $f_i$ as
\begin{displaymath}
    \left\{ \left.  (x^{n}(x^{q_i} - 1) + ( x^{q_i} - 1) \lambda ,q_i) \: \right| \: n \in \mathbb{Z}, \lambda \in \mathbb{F}_p[x, x^{-1}]\right\}.
\end{displaymath}
In order for $f_i \sim g_i \text{ mod } N$, we need to have
\begin{displaymath}
    x-1 + x^{q_i} - 1 \in \{x^n(x^{q_i} - 1) + (x^{q_i} - 1) \lambda \: |\: n \in \mathbb{Z}, \lambda \in \mathbb{F}_p[x, x^{-1}] \} \text{ mod } (P).
\end{displaymath}
The above is equivalent to
\begin{displaymath}
    x - 1 \in \{(x^n + \lambda - 1)(x^{q_i} - 1) \: | \: n \in \mathbb{Z}, \lambda \in \mathbb{F}_p[x, x^{-1}] \} \text{ mod } (P).
\end{displaymath}
Using basic algebra, we see that the above is equivalent to
\begin{displaymath}
    x - 1 \in \{\lambda (x^{q_i} - 1) \: |  \lambda \in \mathbb{F}_p[x, x^{-1}] \} \text{ mod } (P).
\end{displaymath}
Thus, we have that $f_i \sim g_i \text{ mod } N$ if and only if $x - 1 \in (x^{q_i} - 1) \text{ mod } (P)$.

Since $\gcd(x^{q_i} - 1, P) = 1,$ there exist polynomials $\alpha, \beta \in \mathbb{F}_p[X]$ such that
$$
(x^{q_i} - 1) \alpha + P \beta = 1.
$$
By multiplying through by $x-1$, we may write
$$
x - 1 = (x - 1)(x^{q_i} - 1) \alpha + (x-1) P \beta .
$$
Reducing mod $(P)$, we have
$$
x - 1 = (x - 1) \alpha (x^{q_i} - 1)   \text{ mod } (P).
$$ 
We see that $f_i \sim g_i \text{ mod } N,$ and therefore,
$$
p^{q_i} \leq \CD_{\mathbb{F}_p \wr \mathbb{Z}}(f_i, g_i) \leq  q_i p^{q_i - 1}.
$$

From the construction of the elements $f_i, g_i$, it can be easily seen that there is a constant $C'$ such that
\begin{displaymath}
    q_i \leq \|f_i\| \leq C' q_i \quad \text{and} \quad q_i \leq \|g_i\| \leq C' q_i.
\end{displaymath}
There, we have that $p^n \leq \Conj_{\mathbb{F}_p \wr \mathbb{Z}}(C' n)$. Hence, we may write
$$
2^n \preceq \Conj_{\mathbb{F}_p \wr \mathbb{Z}}(n)
$$
since $2^n \approx p^n$.
\end{proof}

\subsection{Lower bounds for $\Conj_{A \wr B}(n)$ when $A$ and $B$ are infinite} $\:$ \newline
 In this subsection, we provide asymptotic lower bounds for $\Conj_{A \wr B}(n)$ where $A$ and $B$ are infinite, finitely generated abelian groups. We start with the group $\mathbb{Z} \wr \mathbb{Z}$ as seen in Proposition \ref{infinite_lower_bound}. Before we start, we have the following lemma.
 
 \begin{lemma}\label{mod_ideal}
     Let $m,n \in \mathbb{N}$ and $d \in \mathbb{Z}$. Then
     $$
     (x^{m} - 1) \equiv ( x^{\gcd(m,n)} - 1 ) \text{ mod } (x^n - 1, d).
     $$
 \end{lemma}
\begin{proof}
We note that $m = \ell \gcd( m,n)$  for some integer $\ell$. Therefore,
$$
x^m \equiv x^{\ell \gcd(m,n)} \equiv 1 \text{ mod } (x^{\gcd(m,n)} - 1 ).
$$
Hence, $x^{\gcd(m,n)} - 1\mid x^{m} - 1 $, and thus,
$$
\left( x^m - 1 \right) \subset (x^{\gcd(m,n)} - 1 ) \text{ mod } (x^n - 1, d).
$$

For the other inclusion, we note that there exist integers $t,s$ such that $\gcd(m,n) = tm + sn.$ Hence, we may write
\begin{eqnarray*}
(x^{tm} - 1)(x^{sn} - 1) =& x^{tm + sn} - x^{tm} - x^{sn} + 1 &\text{ mod } (x^{n} - 1,d) \\
 =& x^{\gcd(m,n)} - x^{sn} -(x^{tm} - 1) &\text{ mod } (x^n - 1, d)\\
\equiv & x^{\gcd(m,n)} - 1 - (x^{tm} - 1 )& \text{ mod } (x^n - 1, d).
\end{eqnarray*}
Since $(x^{tm} - 1)(x^{sn} - 1) \in (x^n - 1,d)$, we have
$$
x^{\gcd(m,n)} - 1 \equiv x^{tm} - 1 \text{ mod } (x^n - 1,d).
$$
Hence,
$$
(x^{m} - 1) \equiv  (x^{\gcd(m,n)} - 1) \text{ mod }(x^n - 1,d). \qedhere
$$
\end{proof}

We now come to the last proposition of this section.

\begin{proposition}\label{infinite_lower_bound}
    Let $A$ and $B$ be infinite, finitely generated abelian groups. Then
    $$
    (\log n)^{n} \preceq \Conj_{A \wr B}(n).
    $$
\end{proposition}
\begin{proof}
Let us first note that we may choose splittings of $A$ and $B$ as direct sums $A \simeq  \mathbb{Z}^k \oplus \Tor(B)$ and $B \simeq \mathbb{Z}^d \oplus \Tor(B)$. Since we assumed that both $A,B$ are infinite, we see that $d,k > 0$. In particular, $A$ contains an element $a$ of an infinite order such that $\langle a\rangle$ is an retract of $A$ and $B$ contains an element $b$ of an infinite order such that $\langle b \rangle$ is an retract of $B$. By Lemma \ref{lemma:conjugacy embedded subgroups} we see that the subgroup $\mathbb{Z} \wr \mathbb{Z} \simeq \langle a \rangle \wr \langle b \rangle$ 
is a retract of $A \wr B$ and $\Conj_{\mathbb{Z} \wr \mathbb{Z}}(n) \preceq \Conj_{A \wr B}(n)$. Hence, we may assume that $A \wr B \cong \mathbb{Z} \wr \mathbb{Z}$.

    We need to find an infinite sequence of pairs of nonconjugate elements $\{f_i, g_i\}$ such that $\log(C \text{max}\{\|f_i\|, \|g_i\|\})^{C \text{max}\{\|f_i\|, \|g_i\|\}} < \CD_{\mathbb{Z} \wr \mathbb{Z}}(f_i, g_i)$ for some $C>0$. For ease of writing, we denote
    \begin{displaymath}
        \mathbb{Z} \wr \mathbb{Z} \simeq \mathbb{Z}[x, x^{-1}] \rtimes \mathbb{Z}
    \end{displaymath}
    where $\mathbb{Z}$ acts by multiplication on $\mathbb{Z}[x, x^{-1}]$ by $x$.

    Let $\{q_i\}$ be an enumeration of the primes, and let $\alpha(i) = \lcm(1, \cdots, q_i - 1)$. Finally, let $k_i$ be the smallest integer such that $\alpha(i) \leq 2^{k_i}$. We define the elements $f_i, g_i \in \mathbb{Z} \wr \mathbb{Z}$ as
    \begin{displaymath}
        f_i = (\alpha(i)(x^{2^{k_i}} - 1),2^{k_i}) \quad \text{ and } \quad g_i= (\alpha(i) (x^{2^{k_i}} - 1 + x^{2^{k_i - 1}} - 1), 2^{k_i}).
    \end{displaymath}
    To see that $f_i$ is not conjugate to $g_i$, we set $\mathfrak{k}_i$ be the ideal in $\mathbb{Z}[x,x^{-1}]$ given by $(2^{k_i}, x^{2^{k_i}}-1)$, and let $H = \mathfrak{k}_i \rtimes  q_i\mathbb{Z} \leq \mathbb{Z} \wr \mathbb{Z}$. We see that $|(\mathbb{Z} \wr \mathbb{Z}) / H | = 2^{k_i} q_i^{2^{k_i}}$ and that
    \begin{displaymath}
        \pi_H (f_i) = (0,0) \quad \text{ and } \quad \pi_H(g_i) = \left(\alpha(i) ( x^{2^{k_i - 1}} - 1), 0 \right) \neq (0,0)
    \end{displaymath}
    where $\pi_H \colon \mathbb{Z} \wr \mathbb{Z} \to (\mathbb{Z} \wr \mathbb{Z})/H$ is the natural projection. Therefore, $f_i \not\sim g_i$ in $\mathbb{Z} \wr \mathbb{Z}$.

    Now suppose that $N \trianglelefteq \mathbb{Z} \wr \mathbb{Z}$ is a finite index subgroup where $|(\mathbb{Z} \wr \mathbb{Z}) / N| < q_i^{2^{k_i}}$. We will show that  $f_i N \sim g_i N$ in $(\mathbb{Z} \wr \mathbb{Z}) / N$.
    
    We note by Lemma \ref{ideal_normal} that $\mathcal{J}_N = N \cap \mathbb{Z}[x,x^{-1}]$ is an ideal in $\mathbb{Z}[x, x^{-1}]$. In particular, $\mathcal{J} \rtimes (N \cap \mathbb{Z})$ is a normal subgroup in $\mathbb{Z} \wr \mathbb{Z}$. Similarly, $N \cap \mathbb{Z} = b \mathbb{Z}$ for some $b \in \mathbb{Z}$. Therefore, we denote $N' =\mathcal{J}_N \rtimes b \mathbb{Z}$. Thus, it follows that $N'$ is a finite index normal subgroup of $\mathbb{Z} \wr \mathbb{Z}$ where $N' \leq N$. In particular, if $f_i N \not\sim g_i N$ in $(\mathbb{Z} \wr \mathbb{Z})/N$, then $f_i N' \not\sim g_i N'$ in $(\mathbb{Z} \wr \mathbb{Z})/N'$. Therefore, we may assume that $(\mathbb{Z} \wr \mathbb{Z})/N$ takes the form $(\mathbb{Z}[x,x^{-1}] / \mathcal{J}) \rtimes ( \mathbb{Z} / b \mathbb{Z})$ where $\mathcal{J}$ is a cofinite ideal and $b$ is an integer.

    Following Lemma \ref{lemma:conjugacy through kernel}, we see that $f_i N'' \nsim g_i N''$ for $N'' = \mathcal{J} \rtimes (b\mathbb{Z} + K)$ where $K \leq \mathbb{Z}$ is the preimage under the projection modulo $b$ of the kernel of the action of $\mathbb{Z}/b \mathbb{Z}$ on $\mathbb{Z}[x, x^{-1}]/\mathcal{J}_N$. Letting $b_0 \geq 0$ be such that $b_0 \mathbb{Z} = b\mathbb{Z} + K$, we note that $N''$ is a finite index normal subgroup where
    \begin{displaymath}
        (\mathbb{Z} \wr \mathbb{Z})/N'' = \left( \mathbb{Z}[x, x^{-1}] \rtimes \mathbb{Z}\right)/(\mathcal{J}_N \rtimes b \mathbb{Z}) \simeq (\mathbb{Z}[x, x^{-1}]/\mathcal{J}_N )\rtimes (\mathbb{Z}/b\mathbb{Z}).
    \end{displaymath}
    Therefore, by the above discussion, we may assume that 
    $$
    (\mathbb{Z} \wr \mathbb{Z})/N \cong (\mathbb{Z}[x,x^{-1}]/ \mathcal{J}) \rtimes (\mathbb{Z} / b \mathbb{Z})
    $$ 
    where $\mathbb{Z} / b \mathbb{Z}$ acts faithfully on $\mathbb{Z}[x,x^{-1}] / \mathcal{J}.$

    We now show we may assume that $\mathbb{Z}/b\mathbb{Z}$ acts freely on $\mathbb{Z}[x,x^{-1}] / \mathcal{J}$. Suppose that there are polynomials  $\rho(x), \lambda(x) \in \mathbb{Z}[x,x^{-1}]/\mathcal{J}$ such that 
    $$
    x^m \rho(x) + \mathcal{J}_N = x^m \lambda(x) + \mathcal{J}_N
    $$
    for some $0 \leq m < b$. Since $x^m$ is a unit, we may cancel and write 
    $$
    \rho(x) +  \mathcal{J}_N = \lambda(x) + \mathcal{J}_N
    $$
    which gives our claim.

    Let $\ell$ be the multiplicative order of $x + \mathcal{J}$ in $\mathbb{Z}[x,x^{-1}] / \mathcal{J}$. We claim that $\ell = b$. By definition, we have that $b$ is the smallest integer such that 
    $$
    x^{b} \rho(x) + \mathcal{J} = \rho(x) + \mathcal{J}
    $$
    for all $\rho(x) \in \mathbb{Z}[x,x^{-1}]$. In particular, we have that $x^{b} \cdot 1  = 1 \mod \mathcal{J}$. Thus, we have that $\ell \mid b$. If $\ell \lneq b$, we then have that $x^\ell\rho(x) = \rho(x) \mod \mathcal{J}$ for all $\rho(x) \in \mathbb{Z}[x,x^{-1}]$. However, that implies $\mathbb{Z} / b \mathbb{Z}$ doesn't act faithfully on $\mathbb{Z}[x,x^{-1}]/\mathcal{J}$ which is a contradiction. Therefore, we have that $\ell = b$.

Since $\mathbb{Z} / b \mathbb{Z}$ acts freely and transitively on the set of powers of $x \text{ mod } \mathcal{J}$ in $\mathbb{Z}[x,x^{-1}] /\mathcal{J}$, we have that $|\mathbb{Z}[x,x^{-1}]/\mathcal{J}| = d^{b}$ where $d$ is the characteristic of the finite ring $\mathbb{Z}[x,x^{-1}]/\mathcal{J}_N$. We note that the ideal $(d, x^{b} - 1)$ is contained in the ideal $\mathcal{J}$ and that $|\mathbb{Z}[x,x^{-1}] / (d, x^{b} - 1)| = d^{b}$. It follows that $\mathcal{J} = (d, x^{b} - 1)$.

    If $d < q_i$, then $d \mid \alpha(i)$, and subsequently, 
    $$
    \alpha(i)(x^{2^{k_i}} - 1), \alpha(i)(x^{2^{k_i}} -1 + x^{2^{k_i - 1}} - 1) \in (d, x^{b_{0}} - 1).
    $$ Hence, $f_i = g_i \text{ mod } N$. Therefore, we may assume that $d \geq q_i.$
    
    By Lemma \ref{conjugacy_class}, we may write the conjugacy class of $f_i$ as
    $$
    \{((x^{\ell} (x^{2^{k_i}} - 1) + (x^{2^{k_i}} - 1)Q, 2^{k_i}) \: | \: \ell \in \mathbb{Z}, Q \in \mathbb{Z}[x,x^{-1}] \}.
    $$
    Thus, we have that $f_i \sim g_i$ if and only if
    $$
    x^{2^{k_i}} - 1 + x^{2^{k_i-1}} -1 \in \{x^{\ell} (x^{2^{k_i}} - 1) + (x^{2^{k_i}} - 1)Q, \: | \: \ell \in \mathbb{Z}, Q \in \mathbb{Z}[x,x^{-1}] \}
    $$
    which is equivalent to
    $$
    x^{2^{k_i-1}} - 1 \in \{ (x^\ell - 1 + Q) (x^{2^{k_i}} - 1) \: | \ell \in \mathbb{Z}, Q \in \mathbb{Z}[x,x^{-1}] \}.
    $$
    Since $x^{\ell} - 1 + Q$ can be any Laurent polynomial, we have that $f_i \sim g_i$ if and only if
    $$
    x^{2^{k_i -  1}} - 1 \in \{ Q (x^{2^{k_i}} - 1) \: | \: Q \in \mathbb{Z}[x,x^{-1}] \}.
    $$
    By Lemma \ref{mod_ideal}, we have that
    $$
    x^{2^{k_i}} - 1 \equiv x^{\gcd(2^{k_i},b)} - 1 \text{ mod } (d, x^b - 1).
    $$
    Therefore, we may write the conjugacy class of $f_i$ in $(\mathbb{Z} \wr \mathbb{Z}) / N$ as
    $$
   \{(Q(x^{2^t} - 1) \: | \: Q \in \mathbb{Z}[x,x^{-1}] \} \text{ mod } (d, x^b - 1).
    $$
    where $0 \leq t < k_i.$ Therefore, $f_i N \sim g_i N$ if and only if $$
    x^{2^{k_i - 1}} - 1 \in \{(Q(x^{2^t} - 1) \: | \: Q \in \mathbb{Z}[x,x^{-1}] \} \text{ mod } (d, x^b - 1).
    $$
    Since $2^{t} \mid 2^{k_i - 1}$, it is well known that $x^{2^t} - 1 \mid x^{2^{k_i-1}} - 1$. Therefore, 
    $$
    x^{2^{k_i - 1}} - 1 \in \{(Q(x^{2^t} - 1) \: | \: Q \in \mathbb{Z}[x,x^{-1}] \} \text{ mod } (d, x^b - 1)
    $$
    when $\gcd(2^{k_i -1}, b) \leq 2^{k_i - 1}$. Hence, if $b \leq 2^{k_i - 1},$ then $\gcd(2^{k_i - 1}, b) \leq 2^{k_i - 1}$, and subsequently, $f_i N \sim g_i N$.  We see that
    \begin{displaymath}
    q_i^{2^{k_i}} < \CD_{\mathbb{Z} \wr \mathbb{Z}}(f_i, g_i) < 2^{k_i} q_i^{2^{k_i}}.
    \end{displaymath}

    Recall that $\alpha(i) = \exp\{\upsilon(q_i - 1)\}$, where $\upsilon \colon \mathbb{N} \to \mathbb{N}$ is the second Chebyshev's function. The Prime Number Theorem \cite[1.2]{number_theory} then implies that there are constants $C_0^-, C_0^+ >0$ such that $2^{C^-_0 q_i} \leq \alpha(i) \leq 2^{C^+_0 q_i}$. Following the definition of $k_i$, we see that there are constants $C_1^-, C_1^+ > 0$ such that $2^{C^-_1 q_i} \leq 2^{k_i} \leq 2^{C^+_1 q_i}$. From the construction of the elements $f_i, g_i$, it can be easily seen that there is a constant $C'$ such that
\begin{displaymath}
    \alpha(i)2^{k_i} \leq n_i \leq C' \alpha(i)2^{k_i},
\end{displaymath}
    where $n_i = \max\{\|f_i\|,\|g_i\|\}$. Following the previous discussion, we see that there are constants $C_2^-,C_2^+ > 0$ such that $2^{C_3^- q_i} \leq n_i \leq 2^{C_3^+ q_i}$. In particular, we see that $q_i \leq \log(Cn_i)$ for some $C >0$. Therefore, $q_i^{2^{k_i}} \geq \log(C^-n_i)^{C^-n_i}$. Thu,s we constructed an infinite sequence of non-conjugate elements $f_i, g_i \in \mathbb{Z} \wr \mathbb{Z}$ that are conjugate in every finite quotient of $\mathbb{Z} \wr \mathbb{Z}$ of size smaller than $\log(C^-n_i)^{C^-\max\{\|f_i\|, \|g_i\|\}}$ where $C^->0$ is some constant. Subsequently, we see that
    \begin{displaymath}
        \log(C^- n_i)^{C^- \max\{\|f_i\|, \|g_i\|\}} < \CD_{\mathbb{Z}\wr\mathbb{Z}}(f_i, g_i).
    \end{displaymath}
    Therefore,
    \begin{displaymath}
         (\log(n))^n \preceq \Conj_{\mathbb{Z} \wr \mathbb{Z}}(n),
   \end{displaymath}
   which concludes the proof.
\end{proof}

\section{Upper bounds}
\label{sec:upper bounds}
The aim of this section is to construct upper bounds for the conjugacy depth function of a wreath product  $A \wr B$ of finitely generated abelian groups. The idea is to show that we can always find a  quotient of the acting group $B$ such that Lemma \ref{lemma:abelian_wreath_conjugacy_criterion} can be used to demonstrate that the images of the elements are not conjugate and provide asymptotic bounds on the size of this quotient. Recall that one of the assumptions of Lemma \ref{lemma:abelian_wreath_conjugacy_criterion} is that we are working with reduced elements, i.e. the elements of the supports lie in distinct cosets of the acting element. Thus, in order to ensure we are working with reduced elements, \ref{subsection:simultaneous cosets} we show how to construct a finite quotient of the acting group that separates finite subsets and infinite cyclic subgroups. Subsection \ref{subsection:translations} then deals with the conditions that Lemma \ref{lemma:abelian_wreath_conjugacy_criterion} uses to establish non-conjugacy. In particular, we show that if a quotient of a finitely generated abelian group is of sufficient size, then certain finite subsets do not become translates of each other in the quotient. Finally, subsection \ref{subsection:uppers} combines these methods to construct a finite quotient preserving non-conjugacy of our given non-conjugate elements and gives an upper bound on its size in terms of their word lengths.

Before we proceed, we recall some notation. If $B$ is a finitely generated abelian group, we by fixing a splitting may write $B = \mathbb{Z}^k \oplus \Tor(B)$ where $\Tor(B)$ is the subgroup of finite order elements of $B$ and $k$ is the torsion-free rank of $B$. Letting $\phi \colon B \to \mathbb{Z}^k$ and $\tau \colon B \to \Tor(B)$ denote the natural projections associated to the fixed splitting, we may then write every $x \in B$ uniquely as $x = \phi(x) + \tau(x)$ where we refer to $\phi(x)$ as the torsion-free part of $x$ and $\tau(x)$ as the torsion part of $x$. When given a vector $b = (b_1, \dots, b_k) \in \mathbb{Z}^k$, we denote $\gcd(b) = \gcd(b_1, \dots, b_k).$ Given two real numbers $a<b$, we let $[a,b]$ denote closed interval from $a$ to $b$. Given two vectors $v,w \in \mathbb{R}^k$, we denote their dot product as $v \cdot w$. Finally, for a finite group $T$, we denote its exponent as $\exp(T).$

\subsection{Simultaneous cosets} $\:$ \newline
\label{subsection:simultaneous cosets}
In this subsection, we study effective separability of cosets of cyclic subgroups in finitely generated abelian groups. Given an infinite, finitely generated abelian group $G$, an element $b \in G$, and a finite subset $S \subseteq B_{G}(\ell)$, we give an upper bound in terms of $\|b\|$ and $\ell$ on the size of a finite quotient of the group $G$ such that each pair of cosets of the cyclic subgroup generated by $b$ corresponding to two distinct elements in $S$ remain distinct. In the following arguments, we use the observation that $s_1 \langle b \rangle = s_2 \langle b \rangle$ if and only if $s_1^{-1} s_2 \in \langle b \rangle$.

The following lemma is important for the proof of Lemma \ref{lemma:torus lemma cosets}.
\begin{lemma}
    \label{lemma:1-dim cosets}
    Let $b \in \mathbb{Z}$ satisfy  $b \in [-n, n]$ and $S \subseteq \mathbb{Z}$ be a subset such that $S \subseteq [-Cn, Cn]$ for some constant $C>0$. Suppose that $c$ is a natural number where $|b|c > CN$, and let $m= 2|b|c$. Finally, let  $\pi \colon \mathbb{Z} \to \mathbb{Z}/m\mathbb{Z}$ be the natural projection. 
    
    Then for every $s \in S$, we have that 
    $$
    \pi(s) \in \langle \pi(b) \rangle \text{ in } \mathbb{Z}/m\mathbb{Z} \quad \text{ if and only if } \quad s \in \langle b \rangle \text{ in }\mathbb{Z}.
    $$
    
    Furthermore, if $\pi(s) \in \langle \pi(b) \rangle$, then $\pi(s) = t \pi(b)$ for the smallest integer $t$ with respect to the absolute value such that $s = t b$. In particular, $|t| \leq m$.
\end{lemma}
\begin{proof}
    Observe that the map $\pi \colon \mathbb{Z} \to \mathbb{Z}/m\mathbb{Z}$ is injective on the interval $[-c|b|-1,c|b|]$. We then note that
    \begin{displaymath}
        \pi(\langle b \rangle) = \pi(b \mathbb{Z}) = \pi\left( \{ -(c-1)|b|, -(c-2)|b|, \dots, -|b|, 0, |b|, \dots, (c-1)|b|,c|b|\}\right).
    \end{displaymath}
    Thus, if $\pi(s) \in \pi(\langle b \rangle)$ for some $s \in S$, then $s \in \langle b \rangle$ since the map $\pi$ is injective on the interval $[-c|b|-1,c|b|]$.
    
    Finally, suppose that $\pi(s) \in \langle \pi(b)\rangle$ and that $\pi(s) = a \pi(b)$ in $\mathbb{Z}/m\mathbb{Z}$ where $a \in \mathbb{Z}$ is the smallest such value with respect to the absolute value. Following the previous argument, it follows that $a b \in [-c|b|-1,c|b|]$, and therefore, we have $s = a b$ in $\mathbb{Z}$.
\end{proof}

To deal with the higher-dimensional cases, we first prove two technical lemmas. This first lemma gives bounds of lengths of a free generating basis for the kernel of the linear map given by the dot product with a vector in terms of size of the entries of the vector. For this lemma, when given vectors $v_1, \ldots, v_k \in \mathbb{Z}^n$, we denote $\langle v_1, \ldots, v_k \rangle$ as the subgroup generated by the set $\{v_1, \ldots, v_k\}.$
\begin{lemma}
    \label{lemma:kernel generators}
    Let $b = (b_1, \dots, b_k) \in \mathbb{Z}^k$ be non-trivial, and let $\varphi_{b} \colon \mathbb{Z}^k \to \mathbb{Z}$ be the homomorphism given by $\varphi_{b}(u) = u \cdot b$. Then there are vectors $\lambda^{(1)}, \dots, \lambda^{(k-1)} \in \mathbb{Z}^k$ such that $\ker(\varphi_{b}) = \langle \lambda^{(1)}, \dots, \lambda^{(k-1)} \rangle$ and $\|\lambda^{(i)}\| \leq 2^{k-1}\|b\|$ for all $i$.
\end{lemma}
\begin{proof}
    We define vectors $b^{(1)}, \dots, b^{(k-1)}$ in the following way:
        \begin{align*}
            b^{(1)}  &= (-b_2, b_1, 0, \dots, 0),\\
                & \vdots  \\
            b^{(k-1)}& = (0, \dots, 0, -b_{k}, b_{k-1}).
        \end{align*}
    We set 
    $$
    \lambda^{(1)} = \frac{1}{\gcd(-b_2, b_1)}b^{(1)},
    $$
    and we note that if $k = 2$, then $\ker(\varphi_{b}) = \langle \lambda^{(1)} \rangle$. Since  $\|\lambda^{(1)}\| \leq \|b\|$, we are done. For $k>2$, we will inductively build a generating set for $\ker(\varphi_{b})$ satisfying the statement of the lemma. We start with some basic observations.
    
    By construction, we have that $b^{(1)}, \dots, b^{(k-1)} \in \ker(\varphi_{b})$. Let $\Lambda_i$ be the maximal subgroup of $\mathbb{Z}^k$ of rank $i$ that contains $b^{(1)}, \dots, b^{(i)}$. Since the vectors $b^{(1)}, \dots, b^{(k-1)}$ are linearly independent over $\mathbb{R}$, we immediately see that $\Lambda_1 \leq \dots \leq \Lambda_{k-1} = \ker(\varphi_b)$ and that $\Lambda_{i}/\Lambda_{i-1} \simeq \mathbb{Z}$ for every $i = 2, \dots, k-1$. 
    
    Now assume that we already have a set of generators for $\Lambda_{i-1}$ which we denote as $\lambda^{(1)}, \dots, \lambda^{(i-1)}$. By construction, the elements $\{\lambda^{(1)}, \dots, \lambda^{(i-1)}\}$ satisfy $\Lambda_j = \langle \lambda_1, \dots, \lambda_j \rangle$ for all $j < i$ where $\|\lambda_j\| \leq 2^{i-1-j} \|b_j\|$ for $1 \leq j \leq i-1.$ Denote $L_i = \langle \Lambda_{i-1}, b^{(i)}\rangle$. Since $\Lambda_{i-1} \leq \langle \Lambda_{i-1}, b^{(i)} \rangle \leq \Lambda_i$, we see that $\Lambda_i/L_i$ is a finite cyclic group. Furthermore, a preimage of some of its generator must be contained within the $i$-dimensional parallelogram given by the vectors $\lambda^{(1)},\dots, \lambda^{(i-1)}, b^{(i)}$. In particular, we see that 
    $$
    \|\lambda^{(i)}\| \leq \|b^{(i)}\| + \sum_{j = 1}^{i-1} \|\lambda^{j}\|.
    $$
 One can then easily check that
    \begin{displaymath}
        \|\lambda^{(i)}\| \leq \|b^{(i)}\| + \sum_{j=1}^{i-1}2^{i-1-j}\|b^{(j)}\|.
    \end{displaymath}
    Noting that $\|b^{(i)}\| = |b_i| + |b_{i+1}|$, we see that $\|\lambda^{(i)}\| \leq 2^{i-1}\|b\|$ as desired.
\end{proof}

This next lemma implies any vector in $\mathbb{Z}^k$ whose entries have greatest common denominator as $1$ is a part of a free base of $\mathbb{Z}^k$. This lemma also shows that there exists an matrix $T \in \GL_k(\mathbb{Z})$ which sends the vector to an element of the canonical basis and gives a bound on how much the matrix $T$ stretches the unit cube in $\mathbb{R}^k$ in terms of the size of the entries of the vector. 
\begin{lemma}
    \label{lemma:stretch factor}
    Let $k > 1$, and suppose that $b = (b_1, \dots, b_k) \in \mathbb{Z}^k$ is a vector where $\gcd(b_1, \ldots, b_k) = 1$. Then $b$ belongs to some free base of $\mathbb{Z}^k$, and moreover, there exists a matrix $T \in \GL_k(\mathbb{Z})$ such that 
    $$
    T(b) = (1,0, \dots, 0) \quad \text{ and } \quad T\left( \prod_{i=1}^k[-1,1] \right) \subseteq \prod_{i=1}^k[2^{k-1}k\|b\|, 2^{k-1}k\|b\|].
    $$
\end{lemma}
\begin{proof}
\cite[Theorem 9]{havas} implies there are integers $a_1, \dots a_k \in \mathbb{Z}$ such that $$\sum_{i=1}^k a_i b_i = \gcd(b_1, \cdots, b_k) = 1$$ and where $\max\{|a_i|\} \leq \frac{1}{2}\max\{|b_i|\}$. Thus, we denote $a = (a_1, \dots, a_k) \in \mathbb{Z}^k$.

Let $\varphi_b \colon \mathbb{Z}^k \to \mathbb{Z}$ be the linear map given by $\varphi_b(x) = x \cdot b$. Lemma \ref{lemma:kernel generators} implies there are vectors 
$$
\lambda^{(1)}, \dots, \lambda^{(k-1)} \in \Ball_{\mathbb{Z}^k}(2^{k-1} n) \subseteq \prod_{i=1}^k[2^{k-1} \|b\|, 2^{k-1} \|b\|]
$$
that freely generate $\ker(\varphi_b)$. We can then form the matrix $T$ by setting the first row to be equal to the vector $a$ and the remaining $k-1$ vectors to be equal to the vectors $\lambda^{(1)}, \dots, \lambda^{(k-1)}$, respectively. By construction, we see that $T(b) = (1,0, \dots, 0)$. Since $\Im(\varphi_b) = \langle \varphi_b(a)\rangle$, we see that $\mathbb{Z}^k = \langle b \rangle \oplus \ker(\varphi_b)$ which implies that the row vectors of $T$ generate $\mathbb{Z}^k$. Therefore, $T \in \GL_k(\mathbb{Z})$.

To finish the proof, we recall that 
$$
a \in \prod_{i=1}^k \left[-\frac{1}{2}\|b\|, \frac{1}{2}\|b\| \right] \quad \text{ and } \quad \|\lambda^{(i)}\| \leq 2^{k-1}\|b\|
$$ 
for all $i$. In particular, this means that for every $i,j \in \{1, \dots, k\}$ we have the $(i,j)$-th entry of $T$ which we denote as $T_{i,j}$ satisfies $|T_{i,j}| \leq 2^{k-1}\|b\|$. Therefore, we have 
\begin{displaymath}
    T\left(\prod_{i=1}^k[-1,1]\right) \subseteq  \prod_{i=1}^k\left[- 2^{k-1}k\|b\|,2^{k-1}k\|b\|\right]. \qedhere
\end{displaymath}
\end{proof}

The following corollary is not consequential for this paper, but we feel it an interesting result in its own right.
\begin{corollary}
    Let $b = (b_1, \dots, b_k) \in \mathbb{Z}$ be given such that $\gcd(b_1, \dots, b_k) = 1$. Then there are elements $\lambda^{(1)}, \dots, \lambda^{(k-1)} \in \mathbb{Z}^k$ such that the set $\{b,\lambda^{(1)}, \dots, \lambda^{(k-1)}\}$ is a free base of $\mathbb{Z}^k$ and $\|\lambda^{(i)}\| \leq 2^k \|b\|$.
\end{corollary}

For a vector $b \in \mathbb{Z}^k$, this next lemma gives bounds on the size of the integer $m$ we reduce entries in $\mathbb{Z}^k$ mod $m$ to preserves cosets of the infinite cyclic subgroup generated by $b$ in terms of the size of the entries in $b$.

\begin{lemma}
    \label{lemma:torus lemma cosets}
    Let $k > 1$ and $n \in \mathbb{N}$ be fixed. Let $S \subseteq \mathbb{Z}^k$ and $b = (b_1, \dots, b_k)  \in \mathbb{Z}^k$ satisfy $b \in \Ball_{\mathbb{Z}^k}(n)$ and where $S \subseteq \Ball_{\mathbb{Z}^k}(Cn)$ for some $C >0$. Suppose $m$ is an integer satisfying the following:
    \begin{enumerate}
    \item $m$ is divisible by $\gcd(b)$; 
    \item $m > 2^{k} k C n^2$;
    \item the homomorphism $\pi \colon \mathbb{Z}^k \to \left(\mathbb{Z}/m\mathbb{Z}\right)^k$ given by reducing every mod $m$ is injective on the set $S$. 
    \end{enumerate}
    Then for every $s \in S$ we have that $\pi(s) \in \langle \pi(b)\rangle$ if and only if $s \in \langle b \rangle$. 
    
    Furthermore, if $\pi(s) \in \langle \pi(b)\rangle$, then there is an integer $t \in \mathbb{Z}$ such that $s = t b$ and $|t| \leq m/c$.
\end{lemma}
\begin{proof}
    Set $b' = \frac{1}{c}b$ and denote $b' = (b_1', \dots, b_k')$. Since $\gcd(b_1', \dots, b_k') = 1$, Lemma \ref{lemma:stretch factor} implies there exists a matrix $T \in \GL_d(\mathbb{Z})$ such that $T(b') = e_1$ and 
    $$
    T \left( \prod_{i=1}^k[-1,1] \right) \subseteq \prod_{i=1}^k[-2^{k-1}k n, 2^{k-1}k n]
    $$
    where $\{e_1, \dots, e_k\}$ is the canonical free basis of $\mathbb{Z}^k$. Since $T$ is an automorphism of $\mathbb{Z}^d$, we see that $T(s) \in \langle T(b) \rangle = \langle c e_1 \rangle$ if and only if $s \in \langle b\rangle$. We note that
    \begin{displaymath}
        T(S) \subseteq T\left( B_{\mathbb{Z}^k}(Cn)\right) \subseteq T\left( \prod_{i=1}^k[-Cn, Cn]\right) \subseteq  \prod_{i=1}^k\left[-2^{k-1} k C n^2, 2^{k-1}k C n^2\right].
    \end{displaymath}
    
    Set $m = 2c l$ where $l \in \mathbb{N}$ is the smallest natural number such that $c l > 2^{k-1}k C n^2$, and denote $K = m \mathbb{Z}^k \leq \mathbb{Z}^k$. By construction, we have that $m \leq 2^{k+1}k C n^2$. Therefore, we see that the projection $\pi_K \colon \mathbb{Z}^k \to \mathbb{Z}^k/ m \mathbb{Z}^k$ is injective on the hypercube $\prod_{i=1}^k[-m+1, m]$ where $\pi_K$ is the reduction of each coordinate mod $m$. In particular, since $S \subseteq \prod_{i=1}^k[-m+1, m]$, for any $s \in S$ we have that $T(s) \left( m \mathbb{Z}^k \right) \subseteq \langle e_1 \rangle K$ if and only if $T(s) \in \langle e_1\rangle$. It then follows that $\pi(s) \notin \langle \pi(b) \rangle$ whenever $s \notin \langle b' \rangle$.
    
    Now suppose that $s \in \langle b' \rangle$, i.e. $T(s) \in \langle e_1 \rangle$. In this case, we may retract onto the first coordinate and assume that we are working in $\mathbb{Z}$. The rest of the statement then follows by Lemma \ref{lemma:1-dim cosets}.
\end{proof}

This next lemma extends Lemma \ref{lemma:torus lemma cosets} to when the abelian group has torsion.
\begin{lemma}
    \label{lemma:cosets_torsion}
    Let $B$ be a finitely generated infinite abelian group of torsion free rank $k$. Let $b \in \Ball_B(n)$ and $S \subset \Ball_{B}(C n )$ be given for some constant $C>0$.
    
    If $k = 1$, assume  that $m \in \mathbb{N}$ satisfies $m \geq 2Cn$ and where both $\|\phi(b)\|$ and $\exp(T)$ divide $m$. If $k \geq 2$, assume that $m \in \mathbb{N}$ is such that $m \geq k2^{k} Cn^2$ and both $c = \gcd(\phi(b))$ and $\exp(\Tor(B))$ divide $m$. Then the homomorphism 
    \begin{displaymath}
        \pi \colon \mathbb{Z}^k \oplus \Tor(B) \to (\mathbb{Z}/m \mathbb{Z})^k \oplus\Tor(B) 
    \end{displaymath}
    defined as the identity on $\Tor(B)$ and as the coordinate-wise projection on $\mathbb{Z}^k$ is injective on the set $S$ and for every $s \in S$ we have that $\pi(s) \in \langle \pi(b)\rangle$ if and only if $s \in \langle b \rangle$.
\end{lemma}
\begin{proof}

    Lemma \ref{lemma:torus lemma cosets} implies we may assume that $\Tor(B) \neq 0$. Therefore, set $e = \exp(\Tor(B))$.
    
    The main argument of the proof when $k=1$ is analogous to the case when $k \geq 2$, but instead of Lemma \ref{lemma:torus lemma cosets} one would use Lemma \ref{lemma:1-dim cosets}. For this reason, we leave proof in the case when $k=1$ as an exercise.
    
    Denote $e = \exp(\Tor(B)),$ and suppose that $m >0$ and $\pi \colon \mathbb{Z}^k \oplus \Tor(B) \to (\mathbb{Z} / m \mathbb{Z})^k \oplus \Tor(B)$ are as in the statement of the lemma. Assuming that $\pi(s) \in \langle \pi(b)\rangle$ for some $s \in S$, there is some $t \in \mathbb{N}$ such that $\pi(s) = t \pi(b)$ which we pick to be as small possible. In particular, we see that $t \leq \gcd(m, e) = m$.
    
    We write:
    \begin{align*}
        \pi(s)                  &= t \pi(b)\\
        \tau(s) + \pi(\phi(s))  &= t \tau(s)) + t\pi(\phi(b)),
    \end{align*}
    from which immediately see that $\tau(s) = t \tau(b)$ and $\pi(\phi(s)) = t\pi(\phi(b))$. Following Lemma \ref{lemma:torus lemma cosets}, we see that $\phi(s) = t \phi(b)$. Therefore, we may write
    \begin{displaymath}
        s = t \tau(b) + t \phi(b) = t b. \qedhere
    \end{displaymath}
\end{proof}

\subsection{Translations} $\:$ \newline
\label{subsection:translations}
 We say that an ordered list $X = (x_1, \dots x_m) \subseteq G^m$ is a \textbf{translate} of an ordered list $Y = (y_1, \dots, y_n) \subseteq G^m$ if $n = m$ and there exists a permutation $\sigma \in \Sym(n)$ such that
\begin{displaymath}
    x_{\sigma(1)}y_1^{-1} = \dots = x_{\sigma(n)}y_n^{-1}.
\end{displaymath}
In this case, we say that $\sigma$ \textbf{realises} a translation of $X$ onto $Y$.

We have the following lemma which gives conditions of when two sets in a product of groups are translates of each other in terms of translations of their images in the projection onto the factor groups.
\begin{lemma}
    \label{lemma:direct product translation criterion}
    Let $G_1, G_2$ be groups, set $G = G_1 \times G_2$ and let $\pi_1, \pi_2$ denote the canonical projections $\pi_1 \colon G_1 \times G_2 \to G_1$ and $\pi_2 \colon G_1 \times G_2 \to G_2$. Suppose that $X,Y \subseteq G$ are two finite subsets where $X = \{x_1, \ldots, x_\ell\}$ and $Y = \{y_1, \ldots, y_m\}$. Then $X$ is a translate of $Y$ if and only if the following two properties hold:
    \begin{enumerate}
    \item $|X| = |Y| = n$
    \item  There exists $\sigma \in \Sym(n)$ such that $\sigma$ simultaneously realises a translation of the list $(\pi_1(x_1), \dots, \pi_1(x_n))$ onto the list $(\pi_1(y_1), \dots, \pi_1(y_n))$ and a translation of the list $(\pi_2(x_1), \dots, \pi_2(x_n))$ onto the list $(\pi_2(y_1), \dots, \pi_2(y_n))$.
    \end{enumerate}
\end{lemma}
\begin{proof}
    It is straightforward to see that $X$ is a translate of $Y$ then $|X| = |Y|$. Thus,  we may assume that $|X|,  |Y| = n$ for some $n \in \mathbb{N}$. As mentioned above, we have that $X$ is a translate of $Y$ if and only if there is $\sigma \in \Sym(n)$ such that
    \begin{displaymath}
        x_{\sigma(1)}y_1^{-1} = \dots = x_{\sigma(n)}y_n^{-1}.
    \end{displaymath}
    If terms of Cartesian coordinates, this means that
    \begin{align*}
        \pi_1(x_{\sigma(1)})\pi_1(y_1^{-1}) = \dots = \pi_1(x_{\sigma(n)})\pi_1(y_n^{-1}),\\
        \pi_2(x_{\sigma(1)})\pi_2(y_1^{-1}) = \dots = \pi_2(x_{\sigma(n)})\pi_2(y_n^{-1}).
    \end{align*}
    That is equivalent to saying that $\sigma$ realises a translation of $(\pi_1(x_1), \dots, \pi_1(x_n))$ onto $(\pi_1(y_1), \dots, \pi_1(y_n))$ and that $\sigma$ realises a translation of the list $(\pi_2(x_1), \dots, \pi_2(x_n))$ onto the list $(\pi_2(y_1), \dots, \pi_2(y_n))$.
\end{proof}

This next lemma tells us there exists a constant $\ell$ such that when given two finite subsets $X, Y$ in $\mathbb{Z}^k$ whose coordinates of each element have absolute value at most $\ell$, then $X$ and $Y$ are translations of each other if and only if their images are translations in the group $(\mathbb{Z} / 4 \ell \mathbb{Z})^k$ where we reduce each coordinate mod $4 \ell.$
\begin{lemma}
    \label{lemma:torus_lemma}
    Suppose that $X = (x_1, \dots, x_n)$ and  $Y = (y_1, \dots, y_n)$ are two finite ordered lists in $\mathbb{Z}^k$ such that 
    $$
    X,Y \subseteq \prod_{i=1}^d[-(\ell-1), \ell-1]
    $$
    for some $\ell \in \mathbb{N}$, and suppose that $c \geq 4 \ell$ Let $\pi \colon \mathbb{Z}^k \to \left(\mathbb{Z}/c \mathbb{Z}\right)^k$ be the homomorphism given by reducing each coordinate mod $c$.
    Then the following are equivalent:
    \begin{enumerate}
    \item $\pi(X) = (\pi(x_1), \dots, \pi(x_n))$ is a translate of $\pi(Y) = (\pi(y_1), \dots, \pi(y_n))$ in $\left(\mathbb{Z}/c \mathbb{Z}\right)^k$.
    \item $X$ is a translate of $Y$ in $\mathbb{Z}^k$.
    \end{enumerate}
    Furthermore, for all $\sigma \in \Sym(n)$, we have that $\sigma$ realises a translation of $\pi(X)$ onto $\pi(Y)$ if and only if $\sigma$ realises a translation of $X$ onto $Y$.
\end{lemma}
\begin{proof}
    We will only prove the `furthermore part of the statement since the first part follows from it.
    
    Since the image of a translate is a translate of an image, the implication from left to right holds trivially. Therefore, we need only consider the other direction. Suppose that $c \geq 4\ell$. For every $\sigma \in \Sym(n)$, define 
    $$
    D_\sigma(X, Y) = \{x_{\sigma(i)} - y_i \mid i = 1, \dots, n\},
    $$ 
    and
    define $D_\sigma(\pi(X), \pi(Y))$ analogously. We observe that $\sigma$ realises a translation of $X$ onto $Y$ if and only if $|D_\sigma(X,Y)|=1$. Similarly, $\sigma$ realises a translation of $\pi(X)$ onto $\pi(Y)$ if and only if $|D_\sigma(\pi(X),\pi(Y))| = 1$. We also have that 
    $$
    D_\sigma(X,Y) \subseteq \prod_{i=1}^k[-2\ell +1, 2\ell -1].
    $$
    We see that $\pi$ is injective on $\prod_{i=1}^k[-2\ell+1, 2\ell-1]$. Hence, $\pi$ is injective on $D_\sigma(X,Y)$ for every $\sigma \in \Sym(n)$. Since $\pi$ is a homomorphism, we see that $D_\sigma(\pi(X), \pi(Y)) = \pi\left(D_\sigma(X, Y)\right)$ for all $\sigma \in \Sym(n)$. Hence, we have that $|D_\sigma(\pi(X), \pi(Y))| = |D_\sigma(X,Y)|$. In particular, we see that 
    $$
    |D_\sigma(\pi(X), \pi(Y))| = 1 \quad \text{ if and only if } \quad |D_\sigma(X,Y)| = 1.
    $$
    Thus, $\sigma$ realises a translation of $\pi(X)$ onto $\pi(Y)$ if and only if $\sigma$ realises a translation of $X$ onto $Y$.
\end{proof}

The last lemma of this subsection extends Lemma \ref{lemma:torus_lemma} to infinite finitely generated abelian groups with torsion.
\begin{lemma}
    Let $B$ be a finitely generated abelian group of torsion-free rank $k$, and suppose that $X, Y \subseteq \Ball_B(\ell)$ are given. If $c \geq 4\ell$, then the homomorphism $\pi \colon B \to \left(\mathbb{Z}/ c \mathbb{Z}\right)^k  \oplus \Tor(B)$ given by the identity on $\Tor(B)$ and by the coordinate-wise projection mod $c$ on the torsion-free part of $B$ is injective on $X \cup Y$.
    
    Moreover, $\pi(X)$ is a translate of $\pi(Y)$ if and only if $X$ is a translate of $Y$ in $\mathbb{Z}^k$. Furthermore, for all permutations $\sigma$, we have that $\sigma$ realises a translation of $\pi(X)$ onto $\pi(Y)$ if and only if $\sigma$ realises a translation of $X$ onto $Y$.
\end{lemma}
\begin{proof}
    \label{lemma:non-translates}
    By assumption, we have 
    $$
    \pi(X), \pi(Y) \subseteq \prod_{i=1}^k[-\ell, \ell]
    $$which implies that $\pi$ is injective on $X \cup Y$. In particular, if $|X|\neq |Y|$, then $|\pi(X)| \neq |\pi(Y)|$, and subsequently, $\pi(X)$ is not a translate of $\pi(Y)$. Therefore, we may assume that $|X| = |Y|$. 
    
    Let $\{x_1, \dots, x_m\}$ and $ \{y_1, \dots, y_m\}$ be enumerations of $X$ and $Y$, respectively. Following Lemma \ref{lemma:direct product translation criterion}, we see that $\pi(X)$ is a translate of $\pi(Y)$ if and only if the list given by $\{\pi(\phi(x_1)),\dots, \pi(\phi(x_m))\}$ is a translation of the list given by $\{\pi(\phi(y_1)),\dots, \pi(\phi(y_m))\}$ and the list $\{\pi(\tau(x_m)),\dots, \pi(\tau(x_m))\}$ is a translation of the list $\{\pi(\tau(y_1)),\dots, \pi(\tau(y_m))\}$ where the translation is realised by the same permutation. However, by Lemma \ref{lemma:torus_lemma}, we see that a permutation $\sigma \in \Sym(m)$ realises a translation of the list given by $\{\pi(\phi(x_1)), \dots \pi(\phi(x_m))\}$ onto the list $\{\pi(\phi(y_1)), \dots \pi(\phi(y_m))\}$ if and only if it realises a translation of the list given by  $\{\phi(x_1), \dots, \phi(x_m)\}$ onto the list given by $\{\phi(y_1), \dots, \phi(y_m)\}$. Since $\pi$ is defined as the identity on $\Tor(B)$, we see that $\pi(X)$ is a translate of $\pi(Y)$ if and only if $X$ is a translate of $Y$, which concludes the proof. 
\end{proof}

\subsection{Finite base groups} $\:$ \newline
\label{subsection:uppers}
Let $A$ be a finite abelian group, and let $B$ be an infinite, finitely generated abelian group. Using Lemma \ref{lemma:cosets_torsion} and Lemma \ref{lemma:non-translates}, the next proposition demonstrates when given non-conjugate elements $x,y$ in a s $A \wr B$ that there exists a finite quotient $\overline{B}$ of $B$ such that the images of $x$ and $y$ in $A \wr \overline{B}$ remain non-conjugate. Moreover, this lemma gives a bound of the size of the quotient of $B$ in terms of the word lengths of $x$ and $y$.

\begin{proposition}
    \label{proposition:abelian_wreath_reduction}
    Let $A$ be an abelian group and $B$ be an infinite, finitely generated abelian group. Let $f,g \colon B \to A$ be finitely supported functions and $b\in B$ an element such that $fb, gb \in \Ball_{A \wr B}(n)$ and $fb \not\sim_{A \wr B} gb$. Then there exists a surjective homomorphism $\pi \colon B \to \overline{B}$ to a finite group such that $\tilde{\pi}(fb) \not\sim \tilde{\pi}(gb)$ in $A \wr \overline{B}$. 
    
    Moreover, there exists a constant $C > 0$ such that if $B$ has torsion free rank $1$, then we have $|\bar{B}| \leq Cn$, and if $B$ is of torsion-free rank $k>1$, we then have $|\bar{B}| \leq C n^{2k}.$
\end{proposition}
\begin{proof}
    Fix a splitting of $B$ into $\mathbb{Z}^k \oplus \Tor(B)$ with associated associated free projection $\phi$. Following Lemma \ref{lemma:reduced_elements}, we may assume that both the functions $f$ and $g$ are given such that the elements $fb$ and $gb$ are reduced, i.e. the individual elements of their respective supports lie in distinct cosets of $\langle b \rangle$ in $B$.
    
    Following Lemma \ref{lemma:abelian_wreath_conjugacy_criterion}, there are two cases to distinguish:
    \begin{enumerate}
        \item[(i)] $\supp(f)$ is not a translate of $\supp(g)$ in $B$,
        \item[(ii)] for every $a \in B$ such that $a + \supp(f) = \supp(g)$, there exists some $x \in \supp(g)$ such that $f(x+a) \neq g(x)$.
    \end{enumerate}
    We will construct a finite quotient $\overline{B}$ such that the images of $fb$ and $gb$ are still reduced in $A \wr \overline{B}$ whether (i) or (ii) is the case.
    
    Lemma \ref{lemma:bounding_support_and_image} implies that there is constant $C_1 > 0$ such that
    \begin{displaymath}
        \{b\} \cup \supp(f) \cup \supp(g) \subseteq B_{B}(C_1 n).
    \end{displaymath}
    In particular, we see that 
    \begin{displaymath}
        \phi(\supp(f)), \phi(\supp(g)) \subseteq B_{\mathbb{Z}^k}(C_1 n) \subseteq \prod_{i=1}^k[-C_1n, C_1n].
    \end{displaymath}
     Set $\ell = C_1 n$. It then follows that 
    $$
    \phi(\supp(f)), \phi(\supp(g)) \subseteq \prod_{i=1}^k[-\ell, \ell].
    $$ 
    We set 
    $$
    S = \{s_2 - s_1 \mid s_1, s_2 \in \supp(f) \cup \supp(g)\}
    $$ and see that $S \subseteq \Ball_B(2\ell)$ and $\|b\| \in \Ball_B(\ell)$. Finally, we set $e = \exp(\Tor(B)).$
    
    If $k =1$, let $m$ be the smallest integer such that $m > 4 \ell$ and where both $e$ and $\|\phi(b)\|$ divide $m$. It is then straightforward to see that $m \leq 8 \ell$. If $k\geq 2$, let $m \in \mathbb{N}$ be smallest possible such that $m > k 2^k 2 \ell^2$ and where both $e$ and $\gcd(\phi(b))$ divide $m$. Without loss of generality, we may assume that $\gcd(\phi(b))$ divides $\|\phi(b)\|$. In particular, we see that $m < k 2^{k+1} 2 \ell^2$.
    
    Via Lemma \ref{lemma:cosets_torsion}, we see that the homomorphism $\pi \colon  \mathbb{Z}^k \oplus \Tor(B)\to (\mathbb{Z}/m \mathbb{Z})^k \oplus \Tor(B)$ defined as the identity on $\Tor(B)$ and as the coordinate-wise reduction mod $m$ on $\mathbb{Z}^k$ is injective on the set $S$ and for every $s \in S$ we have that $\pi(s) \in \langle \pi(b)\rangle$ if and only if $s \in \langle b \rangle$. In particular, this means that for every $s,s' \in \supp(f) \cup \supp(g)$ we have that $\pi(s)\langle \pi(b) \rangle = \pi(s')\langle \pi(b) \rangle$ if and only if $s \langle b \rangle = s'\langle b \rangle$. Set $\overline{B} =  (\mathbb{Z}/m \mathbb{Z})^k \oplus\Tor(B) $, and let $        \tilde{\pi} \colon A \wr B \to A \wr \overline{B}$ be the canonical extension of $\pi$ to the whole of $A \wr B$. From the construction of the map $\pi$, we see that $\tilde{\pi}$ is injective on $\supp(f) \cup \supp(g)$. Therefore, it follows that $\supp(\tilde{\pi}(f)) = \pi(\supp(f))$ and $\supp(\tilde{\pi}(g)) = \pi(\supp(g))$. Furthermore, we see that for every two $s, s' \in \supp(f) \cup \supp(g)$ we have that $\tilde{\pi}(s) \langle \tilde{\pi}(b)\rangle = \tilde{\pi}(s')\langle \tilde{\pi}(b)\rangle$ in $\overline{B}$ if and only if $s\langle b \rangle = s'\langle b \rangle$ in $B$. In particular, we see that the elements $\tilde{\pi}(fb)$ and $\tilde{\pi}(gb)$ are in reduced form.
    
    Since the elements $\tilde{\pi}(fb)$ and $\tilde{\pi}(gb)$ are in reduced form, we may use Lemma \ref{lemma:abelian_wreath_conjugacy_criterion} to check whether or not they are conjugate in $A\wr \overline{B}$. We note that regardless of whether $k = 1$ or $k \geq 2$, we have that $m \geq 4\ell$. Additionally, if $|\supp(f)|\neq |\supp(g)|$, then $|\supp(\tilde{\pi}(f))| \neq |\supp(\tilde{\pi}(g))|$. Subsequently, since $\supp(\tilde{\pi}(f))$ is not a translate of $\supp(\tilde{\pi}(g))$, we have that $\tilde{\pi}(fb)$ is not conjugate to $\tilde{\pi}(fb)$ by Lemma \ref{lemma:abelian_wreath_conjugacy_criterion}. Therefore, we may assume that $|\supp(f)| = |\supp(g)|$. 
    
    
    Let 
    $$
    \supp(f) = \{x_1, \dots, x_m\} \quad \text{ and } \quad \supp(g) =  \{y_1, \dots, y_m\}.
    $$
    Via Lemma \ref{lemma:direct product translation criterion}, we see that $\supp(\tilde{\pi}(f))$ is a translate of $\supp(\tilde{\pi}(g))$ if and only if $\{\pi(\phi(x_1)),\dots, \pi(\phi(x_m))\}$ is a translation of  $\{\pi(\phi(y_1)),\dots, \pi(\phi(y_m))\}$ and the set  $\{\pi(\tau(x_m)),\dots, \pi(\tau(x_m))\}$ is a translation of  $\{\pi(\tau(y_1)),\dots, \pi(\tau(y_m))\}$. Moreover,  the translation of both pairs of sets is realised by the same permutation. Lemma \ref{lemma:torus_lemma} implies that a permutation $\sigma \in \Sym(m)$ realises a translation of $\{\pi(\phi(x_1)), \dots \pi(\phi(x_m))\}$ onto  $\{\pi(\phi(y_1)), \dots \pi(\phi(y_m))\}$ if and only if it realises a translation of  $\{\phi(x_1), \dots, \phi(x_m)\}$ onto the list $\{\phi(y_1), \dots, \phi(y_m)\}$. Since $\pi$ is defined as the identity on $\Tor(B)$, we see that $\supp(\tilde(\pi)(f))$ is a translate of $\supp(\tilde{\pi}(g))$ if and only if $\supp(f)$ is a translate of $\supp(g)$. Therefore, if $\supp(f)$ is not a translate of $\supp(g)$ in $B$, we see by Lemma \ref{lemma:abelian_wreath_conjugacy_criterion} that $\tilde{\pi}(fb)$ is not conjugate to $\tilde{\pi}(gb)$. Thus, we may suppose that $\supp(\tilde{\pi}(f))$ is a translate of $\supp(\tilde{\pi}(g))$.
    
    Now suppose that $a + \supp(f) = \supp(g)$ for some $a \in B$. By assumption, there exists some $x \in \supp(g)$ such that $f(x+a) \neq g(x)$. As mentioned before, Lemma \ref{lemma:torus_lemma} implies that every translation of $\supp(\tilde{\pi}(f))$ onto $\supp(\tilde{\pi}(g))$ must have already occurred in $B$. By the construction of $\pi$, we see that for every $x \in \prod_{i=1}^k[-m, m] \times \Tor(B)$ we have $\tilde{\pi}(f)(x + K) = f(x)$ and $\tilde{\pi}(g)(x + K) = g(x)$. We see that for every $a \in \Tor(B) \times (\mathbb{Z}/m \mathbb{Z})^k$  such that $\pi(a) + \supp(\tilde{\pi}(f)) = \supp(\tilde{\pi}(g))$, there must exist $\overline{x} \in \supp(\tilde{\pi}(f))$ such that $\tilde{\pi}(f)(\pi(a)+\overline{x}) \neq \tilde{\pi}(g)(x)$. That means that $\tilde{\pi}(fb)$ is not conjugate to $\tilde{\pi}(gc)$.
    
    If $k = 1$, set $C = 8  C_1 e |\Tor(B)|$. We then have
    \begin{displaymath}
        \left|\overline{B}\right| = m |\Tor(B)| \leq  \left( 8 C_1 e n \right)|\Tor(B)| = C n.
    \end{displaymath}
    
    If $k \geq 2$, set $C = k^k 2^{k(k+1)} C_1^{2k} e^{2k} |\Tor(B)|$. We then have
    \begin{displaymath}
        \left|\overline{B}\right| = m^k |\Tor(B)| \leq  \left(k 2^{k+1} 2 (C_1 n)^2 \right)^k|\Tor(B)| = C n^{2k}
    \end{displaymath}
    which concludes our proof. 
    \end{proof}

As an immediate consequence of Proposition \ref{proposition:abelian_wreath_reduction}, we get the following upper bound for wreath products of abelian groups with finite base group. 

\begin{proposition}
    \label{lemma:lamplighter_upper}
    Let $A$ is a finite abelian group and $B$ is a finitely generated abelian group of torsion free rank $k$. If $k = 1$,
    $$
    \Conj_{A \wr B}(n) \preceq 2^n.
    $$
    Otherwise, for $k>1$, we have
    \begin{displaymath}
        \Conj_{A \wr B}(n) \preceq 2^{n^{2k}}.
    \end{displaymath}
\end{proposition}
\begin{proof}
    Suppose that $f,b \in A^B$ are finitely supported functions and $b,c \in A$ are elements such that $fb, bc \in B_G(n)$ and where $fb \not\sim_G gc$.
    
    Suppose first that $b \neq c$. Since $b - c \in \Ball_{B}(2n)$, \cite[Corollary 2.3]{BouRabee10} implies there exists a constant $C_1 > 0$ and a surjective homomorphism $\varphi \colon B \to Q$ such that $\varphi(b) \neq \varphi(c)$ and where $|Q| \leq C_1 \: \log(C_1 n).$ Since $Q$ is abelian, we have that $\varphi(b)$ and $\varphi(c)$ are non-conjugate. By composing $\varphi$ with the projection of $A \wr B$ onto $B$, which we also denote $\varphi$, we have a surjective homomorphism $\varphi \colon A \wr B \to Q$ such that $\varphi(fb) \nsim \varphi(gc)$ and where $|Q| \leq C_1  \log(C_1  n).$ Therefore, we may assume that $b = c.$
    
    Following Proposition \ref{proposition:abelian_wreath_reduction}, we have two cases. When the torsion free rank is $1$, we see that there exists a finite abelian group $\overline{B}$ together with a surjective homomorphism $\phi \colon A \wr B \to A \wr \overline{B}$ such that $|\overline{B}| \leq C_2n$ and where $\pi(fb)$ is not conjugate to $\pi(gc)$ in $A \wr \overline{B}$ for some constant $C_2 > 0$. We see that
    \begin{displaymath}
        |A \wr \overline{B}| = |A|^{|\overline{B}|}|\overline{B}| \leq |A|^{C_2n}C_2n. 
    \end{displaymath}
    Interpreting the size of $|A \wr \overline{B}|$ as a function of $n$, we get that
    \begin{displaymath}
        |A \wr \overline{B}| \leq |\overline{B}| |A|^{C_2n}C_2n \preceq |A|^{n} n \preceq |A|^n \preceq 2^n.
    \end{displaymath}
    Subsequently, we see that $\Conj_G(n) \preceq 2^n$.

    When the torsion free rank is greater than $1$, we see that there exists a finite abelian group $\overline{B}$ together with a surjective homomorphism $\phi \colon A \wr B \to A \wr \overline{B}$ such that $|\overline{B}| \leq (C_2n)^{2k}$ and where $\pi(fb)$ is not conjugate to $\pi(gc)$ in $A \wr \overline{B}$ for some constant $C_2 > 0$. We see that $|A \wr \overline{B}| \preceq 2^{n^{2k}}.$ Consequently, we see that $\Conj_G(n) \preceq 2^{n^{2k}}$.
\end{proof}

\subsection{Infinite base groups} $\:$ \newline
When given a wreath product of finitely generated abelian groups $A \wr B$ where $A$ is infinite, the following lemma will allow us to construct an upper bound for size of the quotient of the base group given two elements of $fb$ and $gb$ where $f, g \colon B \to A$ are finitely supported functions.
\begin{lemma}
    \label{lemma:abelian_base_reduction}
    Let $A$ and $B$ be finitely generated abelian groups where $A$ is infinite and $B$ is finite. Let  $f,g \colon B \to A$, $b \in B$ be such that $fb, gb \in \Ball_G(n)$,  the elements $fb$ and $gb$ are reduced, and $fb \not\sim_G gb$. Then there exists a surjective homomorphism $\pi \colon A \to \overline{A}$ to a finite group $\overline{A}$ such that
    \begin{displaymath}
        \left|\overline{A}\right| \leq \min \left\{\log(Cn)^{2|B|}, \log(Cn)^{Cn^2}\right\}
    \end{displaymath}
    and where $\tilde{\pi}(fb) \not\sim \tilde{\pi}(gc)$ in $\overline{A} \wr B$ for some constant $C>0$ independent of $f,g,b,n$.
\end{lemma}
\begin{proof}
    Since $fb \not\sim_G gb$ and the elements $fb, gb$ are reduced, Lemma \ref{lemma:abelian_wreath_conjugacy_criterion} implies that one of the following must be true:
     \begin{enumerate}
        \item $\supp(f)$ is not a translate of $\supp(g)$ in $B$,
        \item for every $a \in B$ such that $a + \supp(f) = \supp(g)$ there exists some $x \in \supp(g)$ such that $f(x+a) \neq g(x)$.
    \end{enumerate}
   Lemma \ref{lemma:bounding_support_and_image} implies there exists a constant $C_1 > 0$ such that 
   $$
   |\supp(f)|, |\supp(g)| \leq C_1 n
   $$ and that $\rng(f), \rng(g) \subseteq B_A(C_1 n)$. Denote $R = \rng(f) \cup \rng(g)$, Clearly $|R| \leq 2C_1 n$ and $R \subseteq B_A(C_1 n)$.

First, we will show that there is a finite group $\overline{A}$ satisfying the requirements on conjugacy such that $|\overline{A}| \leq \log(Cn)^{|B|}$. Suppose that $\supp(f)$ is not a translate of $\supp(g)$ in $B$. Since $A$ is a finitely generated abelian group, its residual finiteness depth function is equivalent to $\log(n)$. It then follows that for every $r \in R \subseteq B$ there is a normal finite index subgroup $K_r$ of $A$ such that $r \notin K_r$ and $|A:K_r| \leq \log(C_1 n)$. Set $K = \cap_{r \in R} K_r$ with natural projection given by $\pi \colon A \to A/K$. As none of the elements in $R$ get mapped to the identity, we see that $\supp(\tilde{\pi}(f)) = \supp(f)$ and $\supp(\tilde{\pi}(g)) = \supp(g)$. In particular, we see that the elements $\tilde{\pi}(fb), \tilde{\pi}(gb)$ are reduced. It follows that $\supp(\tilde{f})$ is not a translate of $\supp(\tilde{g})$. Therefore, we have that $\tilde{\pi}(fb)$ is not conjugate to $\tilde{\pi}(gc)$ in $(A/K) \wr B$ by Lemma \ref{lemma:abelian_wreath_conjugacy_criterion}. We see that
	\begin{displaymath}
		|A/K| \leq \log(C_1 n)^{|R|} \leq \log(C_1 n)^{|B|}.
	\end{displaymath}	
	Suppose that $\supp(f)$ is a translate of $\supp(g)$ and let $T \subseteq B$ be the set of all elements of $B$ that translate $\supp(f)$ onto $\supp(g)$. By assumption, for every $t \in T$ there is $x \in B$ such that $f(t + x) \neq g(x)$. For every such $t$, there exists a normal finite index subgroup $K_t$ of $A$ such that $f(t + x_t)K_t \neq g(x_t)K_t$ for some $x_t \in B$ and $|A: K_t| \leq \log(C_1 n)$. Denote
\begin{displaymath}
	K = \bigcap_{r \in R} K_r \quad \cap \bigcap_{t \in T} K_t,
\end{displaymath}
     where $K_r$ is defined as in the previous paragraph, and let $\pi \colon A \to A/K$ be the natural projection. Clearly, $\supp(\tilde{\pi}(f)) = \supp(f)$ and $\supp(\tilde{\pi}(g)) = \supp(g)$ and, again, we see that the elements $\tilde{\pi}(fb), \tilde{\pi}(gb)$ are reduced. From the construction of the $K$ we see that for every $t \in T$
 there is $x_t \in B$ such that
 \begin{displaymath}
   \tilde{\pi}(f)(x_t + t) = \pi(f(x_t + t)) \neq \pi(g(x_t)) = \tilde{\pi}(g)(x_t).
 \end{displaymath}
That means $\tilde{\pi}(fb)$ is not conjugate to $\tilde{\pi}(gb)$ by Lemma \ref{lemma:abelian_wreath_conjugacy_criterion}. To bound the index of $K$, we can write
\begin{equation}
	\label{eq:cyclic bound}
	|A/K| \leq \log(C_1 n)^{|R|}  \log(C_1 n)^{|T|} \leq  \log(C_1 n)^{2|B|}.
\end{equation}

   Now we show that there is a finite group $\overline{A}$ satisfying the requirements on conjugacy such that $   |\overline{A}| \leq \log(Cn)^{C n^2}.$
   Suppose that $\supp(f)$ is not a translate of $\supp(g)$ in $B$.    Since $A$ is a finitely generated abelian group, \cite[Corollary 2.3]{BouRabee10} implies that for every $r \in R$ there is $K_r \normfileq A$ such that $r \notin K_r$ and $|A/K_r| \leq \log(C_2 n)$ for some constant $C_2 > 0$. Set $K = \cap_{r \in R} K_r$ with natural projection given by $\pi \colon A \to A/K$. As none of the elements in $R$ get mapped to the identity, we see that $\supp(\tilde{\pi}(f)) = \supp(f)$ and $\supp(\tilde{\pi}(g)) = \supp(g)$ and that the elements $\tilde{\pi}(fb), \tilde{\pi}(gb)$ are reduced. It follows that $\supp(\tilde{f})$ is not a translate of $\supp(\tilde{g})$. Therefore, we have that $\tilde{\pi}(fb)$ is not conjugate to $\tilde{\pi}(gc)$ in $(A/K) \wr B$ by Lemma \ref{lemma:abelian_wreath_conjugacy_criterion}. Clearly,
   \begin{displaymath}
       |A/K| \leq \log(C_2 n)^{|R|} \leq \log(C_2 n)^{2C_1 n} \leq \log(2C_2 n)^{2C_1 n}.
   \end{displaymath}
   
   Now suppose for every $a \in B$ such that $a + \supp(f) = \supp(g)$ there exists some $x \in \supp(g)$ such that $f(x+a) \neq g(x)$. For every $r \in R$, let $K_r$ be defined as in the previous paragraph recalling that $|A/K_r| \leq \log(C_0 n)$. For every $\{r,s\} \subset R$ there is $K_{r,s} \normfileq A$ such that $r-s \notin K_{r,s}$ and $|A/K_{r,s}| \leq \log(2C_2 n)$. Denote
   \begin{displaymath}
       K = \bigcap_{r \in R} K_r \quad \cap \bigcap_{\{r,s\} \in R}K_{r,s}
   \end{displaymath}
   with associated natural projection given by $\pi \colon A \to A/K$. Following the same argument as in the previous case, we see that $\supp(\tilde{\pi}(f)) = \supp(f)$, $\supp(\tilde{\pi}(g)) = \supp(g)$, and the elements $\tilde{\pi}(fb), \tilde{\pi}(gb)$ are reduced. Now suppose that $a \in B$ is given such that $a + \supp(\tilde{\pi}((f)) = \supp(\tilde{\pi}(g))$. That is equivalent to $a + \supp(f) = \supp(g)$. Thus, there is some $x_a \in B$ such that $f(x_a + a) \neq g(x_a)$. However, from the construction of $K$ we see that 
   $$
   \tilde{\pi}(f)(x_a + a) = \pi(f(x_a + a)) \neq \pi(g(x_a)) = \tilde{\pi}(g)(x_a).
   $$ 
   In particular, Lemma \ref{lemma:abelian_wreath_conjugacy_criterion} implies that $\tilde{\pi}(fb)$ is not conjugate to $\tilde{\pi}(g)$ in $(A/K)\wr B$. Finally, we finish with
   \begin{equation}	
  	\label{eq:higher rank bound}
       |A/K| \leq \log(C_2 n)^{|R|} \cdot \log(2 C_2 n)^{\left| \binom{R}{2}\right|} \leq \log(2 C_1 n)^{2 \binom{2 C_2 n}{2}} \leq \log(4C_2 n)^{4C_2 n^2}.
   \end{equation}
   Combining (\ref{eq:cyclic bound}) and (\ref{eq:higher rank bound}) immediately yields the result.
\end{proof}

Combining Proposition \ref{proposition:abelian_wreath_reduction} together with Lemma \ref{lemma:abelian_base_reduction} gives the following upper bound for wreath products of infinite, finitely generated abelian groups.
\begin{proposition}\label{infinite_upper_bound}
    Suppose that $A,B$ are infinite finitely generated abelian groups. If $B$ is virtually cyclic then
    \begin{displaymath}
        \Conj_{A \wr B}(n) \preceq (\log(n))^{n^2}.
    \end{displaymath}
    Otherwise,
    \begin{displaymath}
        \Conj_{A \wr B}(n) \preceq (\log(n))^{n^{2k+2}}
    \end{displaymath}
    where $k$ is the torsion-free rank of $B$.
\end{proposition}
\begin{proof}
    Let $k$ denote the torsion-free rank of $B$, and suppose that $n \in \mathbb{N}$, $f,g \in A^B$, and $b,c \in B$ are given such that $fb, gc \in B_{A \wr B}(n)$ and $fb \not\sim gc$.
    
    Suppose first that $b \neq c$. Since $b - c \in \Ball_{B}(2n)$, \cite[Corollary 2.3]{BouRabee10} implies there exists a constant $C_1 > 0$ and a surjective homomorphism $\varphi \colon B \to Q$ such that $\varphi(b) \neq \varphi(c)$ and where $|Q| \leq C_1 \log(C_1 n).$ Since $A$ is abelian, we have that $\varphi(b)$ and $\varphi(c)$ are non-conjugate. By composing $\varphi$ with the projection of $A \wr B$ onto $B$ which we also denote $\varphi$, we have a surjective homomorphism $\varphi \colon A \wr B \to Q$ such that $\varphi(fb) \nsim \varphi(gc)$ and where $|Q| \leq C_1 \: \log(C_1 n).$ Thus, we may assume that $b = c.$
    
    Following Lemma \ref{lemma:reduced_elements} there exist a constant $C_0$ and functions $f',g' \in A^B$ such that $f'b \sim fb$ and $g'b \sim gb$, the elements $f'b, g'b$ are reduced, and $fb, gb \in B_{A \wr B}(C_0n)$.
    
    Following Proposition \ref{proposition:abelian_wreath_reduction}, there is a constant $C_B$ (independent of $n,f,g,b,c$) and a finite abelian group $\overline{B}$ together with a surjective homomorphism $\pi_B \colon B \to \overline{B}$ such that the elements $\tilde{\pi_B}(f'c)$ and $\tilde{\pi_B}(g'c)$ are reduced, $\tilde{\pi_B}(f'b)$ is not conjugate to $\tilde{\pi_B}(g'c)$ in $A \wr \overline{B}$, where $|\overline{B}| \leq C_0 C_B n$ if $B$ has torsion free rank $1$ and where $|\overline{B}| \leq C_0 C_B n^{2k}$ if $B$ has torsion free rank $k>1$.

    Set $\overline{G} = A \wr \overline{B}$. One can easily check that $\tilde{\pi}_B(B_G(C_0n)) \subseteq B_{\overline{G}}(C_0n)$. In particular, $\tilde{\pi}_B(f'b), \tilde{\pi}_B(g'c) \subseteq B_{\overline{B}}(C_0n)$. Lemma \ref{lemma:abelian_base_reduction} implies that there is a constant $C_A$ (also independent of $n,f,g,b,c$) and a finite abelian group $\overline{A}$ together with a surjective homomorphism $\pi_A \colon A \to \overline{A}$ such that 
    \begin{displaymath}
    	|\overline{A}| \leq \min\left\{\log(C_A n)^{2|B|}, \log(C_A n)^{C_An^2}\right\}
    \end{displaymath}
    and where $\tilde{\pi}_A(\tilde{\pi}_B(fb))$ is not conjugate to $\tilde{\pi}_A(\tilde{\pi}_B(gc))$ in $\overline{A} \wr \overline{B}$.

    Therefore,  $\overline{A} \wr \overline{B}$ is a finite group. If $B$ has torsion free rank $1$, then $|B| \leq C_0 C_B n$. If $B$ is virtually cyclic, we set $C = 2 C_0 C_B^2$ and compute
    \begin{align*}
        \left|\overline{A}\wr \overline{B}\right| = \left|\overline{B}\right| \cdot \left|\overline{A}\right|^{\left|\overline{B}\right|} &\leq C_0 C_B n \cdot \left(\log(C_A n)^{ 2 C_0 C_B n}\right)^{C_0 C_B n}\\
        &= C_0 C_B n \cdot \left(\log(C_A n)\right)^{2 C_0^2 C_B^2 n^{2}}\\
        &\leq Cn \left(\log(Cn)\right)^{C n^{2}}.
    \end{align*}
    Interpreting $\left|\overline{A}\wr \overline{B}\right|$ as a function of $n$ immediately yields
    \begin{displaymath}
        \left|\overline{A}\wr \overline{B}\right| \leq Cn \left(\log(Cn)\right)^{C n^{2}} \approx n \log(n)^{n^{2}} \approx \log(n)^{n^{2}}.
    \end{displaymath}
    Therefore, we have that
    \begin{displaymath}
        \Conj_G(n) \preceq \log(n)^{n^{2}}.
    \end{displaymath}
    
    Now suppose that $k > 1$. Set $C = C_0 C_A C_B$. We then compute
    \begin{align*}
        \left|\overline{A}\wr \overline{B}\right| = \left|\overline{B}\right| \cdot \left|\overline{A}\right|^{\left|\overline{B}\right|} &\leq C_0 C_B n^{2k} \cdot \left(\log(C_A n)^{C_A n^2}\right)^{C_0C_B n^{2k}}\\
        &=C_0 C_B n^{2k} \cdot \left(\log(C_A n)\right)^{C_0 C_A C_B  n^{2k+2}}\\
        &\leq Cn^{2k} \left(\log(Cn)\right)^{(C n)^{2k+2}}.
    \end{align*}
    Interpreting $\left|\overline{A}\wr \overline{B}\right|$ as a function of $n$ immediately yields
    \begin{displaymath}
        \left|\overline{A}\wr \overline{B}\right| \leq Cn^{2k} \left(\log(Cn)\right)^{C n^{2k+2}} \approx n^{2k} \log(n)^{n^{2k+2}} \approx \log(n)^{n^{2k+2}},
    \end{displaymath}
    and therefore,
    \begin{displaymath}
        \Conj_G(n) \preceq \log(n)^{n^{2k+2}},
    \end{displaymath}
    which concludes our proof.              
\end{proof}

\section{Applications}
\label{sec:proofs}

In this section, we use Corollary \ref{corollary:main_corrolary} to derive lower bounds for the conjugacy separability depth function where the acting group is not necessarily abelian. For the statement of the theorem, we denote the center of a group $G$ as $Z(G)$.

\begin{theorem}\label{second_main_thm}
Let $A$ be a nontrivial finitely generated abelian group, and suppose that $G$ is a conjugacy separable finitely generated group with separable cyclic subgroups. Suppose that $\mathbb{Z} \leq Z(G)$. If $A$ is finite, then
$$
2^n \preceq \Conj_{A \wr G}(n).
$$
Otherwise,
$$
(\log(n))^{n}  \preceq \Conj_{A \wr G}(n).
$$
\end{theorem}
\begin{proof}

Let $B = \mathbb{Z}^k \leq Z(G)$. We claim that if $(f,b), (g,b) \in A \wr B$ then $x \sim_{A \wr G} y$ if and only if there exists $z \in A \wr B$ such that $z \: x \: z^{-1} = y$. Suppose first that $x \sim_{A \wr G} y$.  \cite[Lemma 5.13]{ferov_pengitore} implies we may assume that elements of $\supp(f)$ (respectively $\supp(g)$) lie in different right cosets of $\left< b\right>$. \cite[Lemma 5.14]{ferov_pengitore} implies that $x \sim_{A \wr G} y$ if and only if there exists an element $c \in C_G(b)$ such that $c f c^{-1} = g.$ That implies for all for $x \in G$, we have that $f(cx) = g(x)$. We note that $g(x) \neq 0$ only if $x \in B$. That implies $f(cx) \neq 0$ only if $cx \in B$. Hence, $c \in B.$ In particular, if $(h, c) \in A \wr B$ where $c \notin B$, then $(h,c) \cdot x \cdot (h,c)^{-1} \neq y$. Thus, if $(h, c) \cdot x \cdot (h,c)^{-1} = y$, then $c \in B$. Hence, we may write
\begin{eqnarray*}
(h, c) (f,b) (h, c)^{-1} &=& (h, c) (f,b)( (c^{-1} \cdot h,  c^{-1} ) \\
&=& (h + c \cdot f, c b) (c^{-1} \cdot h,  c^{-1} )\\
&=& (h + c\cdot f + (c b) \cdot c^{-1} \cdot h, g)\\
&=& ( h + c\cdot f + b \cdot h, g).
\end{eqnarray*}
Thus, if $(h,c) \cdot x \cdot (h, c) = y$, we must have that
$$
\supp(h + g^\prime \cdot f + g \cdot h) \subseteq B.
$$
Suppose $\supp(h) \not \subseteq B$.
In this case, we note that $\supp(h + g \cdot h) \not \subseteq B$ and that $\supp(c\cdot f) \subseteq B$. Therefore, $\supp(h + g^\prime \cdot f + g \cdot h) \not \subseteq B.$ Hence, we have that $(h,c) \cdot x \cdot (h,c)^{-1} \neq y$ which is a contradiction. Therefore, $\supp(f) \subseteq B$ which implies $(h,c) \in A \wr B.$ Since the other direction is clear, we have our claim.

    Subsequently, we have $\Conj_{A \wr B}(n) \preceq \Conj_{A \wr G}(n).$ Our theorem then follows from Corollary \ref{corollary:main_corrolary}.
\end{proof}

Since the rank of the center of an infinite, finitely generated nilpotent group is always positive, we have the following corollary.
\begin{corollary}
    Let $A$ be a finitely generated abelian group, and let $N$ be an infinite, finitely generated nilpotent group. If $A$ is finite, then
    $$
    2^n \preceq \Conj_{A \wr N}(n).
    $$
   Otherwise,
    $$
    (\log(n))^{n} \preceq \Conj_{A \wr N}(n).
    $$
    In both cases, we have that $\Conj_{A \wr N}(n)$ has at least exponential growth.
\end{corollary}

Finally, we consider the case when the acting group contains the integers as a retract.
\begin{theorem}
    \label{theorem:acting with cyclic retract}
    Let $A$ be a nontrivial finitely generated abelian group, and suppose that $G$ is a conjugacy separable finitely generated group with separable cyclic subgroups that contains an infinite cyclic group as a retract. If $A$ is finite, then
$$
2^n \preceq \Conj_{A \wr G}(n).
$$
Otherwise,
$$
(\log(n))^{n}  \preceq \Conj_{A \wr G}(n).
$$ 
\end{theorem}
\begin{proof}
    Let $g \in G$ be an element of infinite order such that $\langle g \rangle$ is a retract in $G$. Then by Lemma \ref{lemma:conjugacy embedded subgroups} we see that $A \wr \langle g \rangle \simeq A \wr \mathbb{Z}$ is a retract in $A \wr G$ and $\Conj_{A \wr \langle g \rangle}(n) \leq \Conj_{A \wr G}(n)$. The rest then follows by Corollary \ref{corollary:main_corrolary}.
\end{proof}

\begin{corollary}
    Let $A$ be a finitely generated abelian group, and suppose that $G$ belongs to one of the following classes of groups:
    \begin{enumerate}
        \item[(i)] right-angled Artin groups,
        \item[(ii)] infinite finitely generated nilpotent and polycyclic groups,
        \item[(iii)] limit groups,
        \item[(iv)] fundamental groups of hyperbolic fibered 3-manifolds,
        \item[(v)] graph products of any of the above.
    \end{enumerate}
    If $A$ is finite, then
    \begin{displaymath}
        2^n \preceq \Conj_{A \wr G}(n).
    \end{displaymath}
    Otherwise,
\begin{displaymath}
    (\log(n))^{n}  \preceq \Conj_{A \wr G}(n).
\end{displaymath}
\end{corollary}
\begin{proof}
    Right-angled Artin groups were shown to be conjugacy separable by Minasyan in \cite[Theorem 1.1]{ashot_raags}, cyclic subgroup separable by Green in \cite[Theorem 2.16]{green}, and checking that a right-right angled Artin group admits an infinite cyclic retract is easy - just consider an endomorphism that maps all but one generator to the identity.
    
    For fundamental groups of closed orientable surfaces, they are known to be subgroup separable due to Scott \cite{Scott1, Scott2} and basic structure theory of abelian groups. The fact that they are conjugacy separable follows from Martino \cite{armando}. Lastly, it is clear that they have infinite cyclic retracts due to the fact that they have infinite abelianizations.
    
    Infinite polycyclic and finitely generated nilpotent groups are well known to conjugacy separable and subgroup separable. Proofs of these results can be found in \cite[Theorem 3, pg59]{sega83-1} and \cite{malcev}. The fact that these classes of groups admit an infinite cyclic retract follows from the fact that infinite polycyclic groups and infinite finitely generated abelian groups always have infinite abelianization by basic structure results found in \cite{sega83-1}.
    
    For limit groups, it was shown by Wilton that they are subgroup separable in \cite[Theorem A]{limit}, and Chagas and Zalesskii demonstrated that they are conjugacy separable in \cite[Theorem 1.1]{limit}. The fact that they have infinite cyclic retracts follows from the fact that limit groups are fully residually free.
    
    The proof that fundamental groups of hyperbolic $3$-manifold groups are conjugacy separable follows from Hamilton, Wilton, and Zalesskii \cite[Theorem 1.3]{compact} and subgroup separability follows from \cite[Corollary 4.2.3]{3_manifold_groups}. The fact that they have infinite cyclic retracts follows from these fact that these groups have the form $\pi_1(\Sigma_g) \rtimes \mathbb{Z}$ where $\Sigma_g$ is a closed orientable genus $g\geq 2$ surface.
    
    The class of conjugacy separable groups is closed under forming graph products by \cite[Theorem 1.1]{mf}. The class of cyclic subgroup separable groups is closed under forming graph products by \cite[Theorem A]{berlai-mf-cyclic}. Finally, to see that a graph product of groups that admits an infinite cyclic retract again admits an infinite cyclic retracts easily follows from the fact that all vertex groups are themselves retracts.
\end{proof}

The vast range of examples we were able to construct in this paper either with Theorem \ref{second_main_thm} and Theorem \ref{theorem:acting with cyclic retract} lead us to believe that the lower bounds we produced cannot be relaxed and therefore we state the following conjecture.
\begin{conjecture}
    Let $A$ be a finitely generated abelian group and let $G$ be a conjugacy separable group with separable cyclic subgroups. If $A$ is finite, then
$$
2^n \preceq \Conj_{A \wr G}(n).
$$
Otherwise,
$$
(\log(n))^{n}  \preceq \Conj_{A \wr G}(n).
$$
\end{conjecture}

\section{Open questions}
The constructions of lower bounds for $\mathbb{F}_p \wr \mathbb{Z}$ and $\mathbb{Z} \wr \mathbb{Z}$ given in Section \ref{sec:lower bounds} relies heavily on the fact that both $\mathbb{F}_p \wr \mathbb{Z}$ and $\mathbb{Z} \wr \mathbb{Z}$ can be represented as a semidirect product of the additive group of the ring of Laurent polynomials over $\mathbb{F}_p$ and $\mathbb{Z}$, respectively, and that subgroups of finite index correspond to cofinite ideals in the said polynomial rings. This approach can be generalised: let $R$ denote either $\mathbb{F}_p$ or $\mathbb{Z}$, then
\begin{displaymath}
    R \wr \mathbb{Z}^k \simeq R[X_1, X_1^{-1}, \dots, X_k, X_k^{-1}] \rtimes \mathbb{Z}^k,
\end{displaymath}
where the action of $\mathbb{Z}^k$ on $\simeq R[X_1, X_1^{-1}, \dots, X_k, X_k^{-1}]$ is given by
\begin{displaymath}
    (b_1, \dots, b_k) \cdot X_1^{e_1} \dots X_k^{e_k} = X_1^{b_1 + e_1} \dots X_k^{e_k + b_k}.
\end{displaymath}
One can easily check that finite index subgroups of $R \wr \mathbb{Z}^k$ correspond to cofinite ideals of $R[X_1, X_1^{-1}, \dots, X_k, X_k^{-1}]$ and statement similar to Lemma \ref{conjugacy_class} can be proved, but the arguments based on divisibility in $R[X,X^{-1}]$ do not carry over to $R[X_1, X_1^{-1}, \dots, X_k, X_k^{-1}]$.
\begin{question}
    Can the proofs of Proposition \ref{rank_one_lamplighter} and Proposition \ref{infinite_lower_bound} be modified to produce lower bounds for the conjugacy depth functions of $\mathbb{F}_p \wr \mathbb{Z}^k$ and $\mathbb{Z} \wr \mathbb{Z}^k$, respectively, that dominate and are not dominated by $2^n$ and $\log(n)^n$, respectively? In particular, can it be shown that
    \begin{displaymath}
        2^{n^k} \preceq \Conj_{\mathbb{F}_p \wr \mathbb{Z}^k}(n)
    \end{displaymath}
    and
    \begin{displaymath}
        \log(n)^{n^k} \preceq \Conj_{\mathbb{Z} \wr \mathbb{Z}^k}(n)?
    \end{displaymath}
\end{question}

The upper bounds given by Proposition \ref{lemma:lamplighter_upper} and Proposition \ref{infinite_upper_bound} treat wreath products of abelian groups differently, based on the torsion-free rank of the acting group, which might feel somewhat disappointing. The following question is therefore a naturally arising one. 
\begin{question}
     Can the proofs of Proposition  Proposition \ref{lemma:lamplighter_upper} and Proposition \ref{infinite_upper_bound} be modified to produce upper bounds for the conjugacy depth functions of $\mathbb{F}_p \wr \mathbb{Z}^k$ and $\mathbb{Z} \wr \mathbb{Z}^k$, respectively, that are given by a single closed-form formula that does not have cases based on the torsion-free rank of the acting group? In particular, can it be shown that
     \begin{displaymath}
         \Conj_{\mathbb{F}_p \wr \mathbb{Z}^k}(n) \preceq 2^{n^k}
     \end{displaymath}
     and
     \begin{displaymath}
         \Conj_{\mathbb{Z} \wr \mathbb{Z}^k}(n) \preceq \log(n)^{n^{2k}}?
     \end{displaymath}
\end{question}

\bibliographystyle{plain}
\bibliography{references}

\end{document}